\RequirePackage[l2tabu, orthodox]{nag}
\RequirePackage{snapshot}

\documentclass[10pt,onecolumn]{extarticle}

\makeatletter
\if@twocolumn
  \usepackage[dvips,letterpaper,top=0.5in, bottom=0.5in, left=0.75in, right=0.5in,includefoot,heightrounded]{geometry}
\else
  \usepackage[dvips,a4paper,margin=1in,includefoot,heightrounded]{geometry}
\fi

\usepackage{srcltx}

\usepackage[russian,portuges,english]{babel}

\iflanguage{portuges}
    {\newcommand{\keywordname}{Palavras-chaves}}
    {\newcommand{\keywordname}{Keywords}}

\usepackage{amsmath}
\usepackage{amssymb,amsfonts}

\usepackage{abstract}

\usepackage{graphicx}
\usepackage{epsfig}
\usepackage[usenames,dvipsnames,svgnames,x11names]{xcolor}
\usepackage{subfigure}

\usepackage{psfrag}

\usepackage{booktabs}

\usepackage{setspace}
\usepackage{flushend}

\usepackage{cite}

\usepackage{hyperref}
\usepackage[normalem]{ulem}

\usepackage{enumerate}

\usepackage{multirow}
\usepackage[noend]{algpseudocode}

\usepackage{listings}

\usepackage[short,12hr]{datetime}
       \usepackage{fouriernc}
\newtheorem{proposition}{Proposition}

\newtheorem{lemma}{Lemma}

\makeatletter

\newcommand{\printtitle}{%
\makeatletter
\if@twocolumn

\twocolumn[%
  \maketitle
  \begin{onecolabstract}
    \myabstract
  \end{onecolabstract}
  \begin{center}
    \small
    \textbf{\keywordname}
    \\\medskip
    \mykeywords
  \end{center}
  \bigskip
]
\saythanks
\else
  \maketitle
  \begin{onecolabstract}
    \myabstract
  \end{onecolabstract}
  \begin{center}
    \small
    \textbf{\keywordname}
    \\\medskip
    \mykeywords
  \end{center}
  \bigskip
  \onehalfspacing
\fi
\makeatother
}

\author{%
R.~J.~Cintra%
\thanks{%
Departamento de Estat\'istica, Universidade Federal de Pernambuco,
Recife, PE, Brazil.
E-mail: \url{rjdsc,gauss,abraao@de.ufpe.br}}
\and
L. C. R\^ego%
\thanks{Formely Departamento de Estat\'istica, Universidade Federal de Pernambuco, Recife, Brazil; currently
Departamento de Estat\'{\i}stica e Matem\'atica Aplicada,
Universidade Federal do Cear\'a, Brazil.
E-mail: \url{leandro@dema.ufc.br}}
\and
G.~M.~Cordeiro${}^\ast$
\and
A.~D.~C.~Nascimento${}^\ast$
}

\title{%
Beta Generalized Normal Distribution with an Application for SAR Image Processing}

\newcommand{\myabstract}{%
We introduce the beta generalized normal distribution
which is obtained by compounding the beta and generalized normal
[Nadarajah, S.,
A generalized normal distribution,
\emph{Journal of Applied Statistics}.
32, 685--694, 2005]
distributions. The new model includes as sub-models the beta normal,
beta Laplace, normal, and Laplace distributions. The shape of the new
distribution is quite flexible, specially the skewness and the tail
weights, due to two additional parameters. We obtain general
expansions for the moments.
The estimation of the parameters is investigated by
maximum likelihood.
We also proposed a random number generator for the new distribution.
Actual synthetic aperture radar were analyzed and modeled
after the new distribution.
Results could outperform
the $\mathcal{G}^0$, $\mathcal{K}$, and $\Gamma$ distributions
in several scenarios.
}

\newcommand{\mykeywords}{%
Beta normal distribution; generalized normal distribution; hazard function; maximum likelihood estimation; moment; SAR imagery
}

\date{}

\begin{document}

\printtitle

\section{Introduction}

Sonar, laser, ultrasound B-scanners, and synthetic aperture radar (SAR) are
sensing devices which employ coherent illumination for
imaging purposes~\cite{UlabyElachi1990}.
Due to its capability of operating in any weather condition and
providing high spatial image resolution,
these microwave active sensors have been considered as
powerful remote sensing techniques~\cite{Ulabyetal1986a}.
In general terms, the operation of these systems consists of
the emission of orthogonally polarized pulses transmitted towards a target
and
the recording of the returned echo signal.

Due to its acquisition process by means of coherent illuminations, SAR images are strongly contaminated by a particular signal-dependent granular noise called speckle~\cite{NascimentoCintraFreryIEEETGARS}.
Thus,
adequate statistical modeling for this type of imagery is constantly sought~\cite{DoulgerisandEltoft2010}.
Several distributions have been proposed in the literature for modeling intensity SAR data.
According to Frery~\textit{et al.}~\cite{freryetal1997a}, the $\mathcal{G}^0$, $\mathcal{K}$, and $\Gamma$ distributions are the most prominent models for SAR data.

Indeed, a SAR image can be understood as a set of regions described by different probability distributions.
Some researchers have devoted considerable attention to study the beta distribution as a model for SAR images.
For instance, in~\cite{ElZaartZiou2007} the beta
distribution was suggested as an important element for modeling SAR images with multi-modal histogram~\cite{DelignonPieczynski2002}.

Eugene~\emph{et al.}~\cite{eugene2002beta} introduced  a class of distributions
generated from the logit of the beta random variable.
For a given baseline cumulative
distribution function (\mbox{cdf}) $G(x)$, the associated beta generalized
distribution $F(x)$ is defined by
\begin{equation}\label{cdfbetafamily}
F(x)
=
I_{G(x)}(\alpha,\beta)
=
\frac{1}{\mathrm{B}(\alpha,\beta)}
\int_0^{G(x)}t^{\alpha-1}(1-t)^{\beta-1}\mathrm{d}t.
\end{equation}
Here, $\alpha>0$ and $\beta>0$ are two additional parameters which aim to introduce
skewness and to vary tail weights,
$I_y(\alpha,\beta)=\mathrm{B}_y(\alpha,\beta)/\mathrm{B}(\alpha,\beta)$
is the incomplete beta function ratio,
$\mathrm{B}_y(\alpha,\beta)=\int_0^y w^{\alpha-1}(1-w)^{\beta-1}\mathrm{d}w$
is the incomplete beta function,
$\mathrm{B}(\alpha,\beta)=\Gamma(\alpha)\Gamma(\beta)/\Gamma(\alpha+\beta)$
is the beta function and $\Gamma(\cdot)$ is the gamma function.
This class
of beta generalized distributions has been receiving considerable attention
after the work
of Eugene \emph{et al.}~\cite{eugene2002beta}.

The beta distribution is a basic exemplar of (\ref{cdfbetafamily}) by
taking $G(x)=x$.
Although it has only two parameters, the beta density accommodates
a very wide variety of shapes including the standard uniform distribution
for $\alpha=\beta= 1$.
The beta density is symmetric, unimodal, and bathtub shaped
for $\alpha=\beta$, $\alpha,\,\beta > 1$ and $\alpha,\,\beta < 1$, respectively.
It has positive skew when $\alpha<\beta$ and negative skew when $\alpha>\beta$.

Some special beta generalized distributions were developed in recent years.
To cite a few, we identify
the beta normal distribution~\protect{\cite{eugene2002beta}},
the beta Gumbel distribution~\protect{\cite{nadarajah2004gumbel}},
the beta Frechet distribution~\protect{\cite{NadarajahGupta}},
the beta exponential distribution~\protect{\cite{nadarajah2005exponential}},
the beta Weibull distribution~\protect{\cite{lee2007censored}},
the beta Pareto distribution~\protect{\cite{Akinseteetal}},
the beta generalized exponential distribution~\protect{\cite{BarretoSouzaetal}},
and
the beta generalized half-normal distribution~\protect{\cite{Pescimetal}}.

The first goal of this note
is to develop an extension of the generalized
normal (GN) distribution defined from~\eqref{cdfbetafamily}:
the beta generalized normal (BGN) distribution.
It may be mentioned
that although several skewed distribution functions exist on the positive
real axis, but not many skewed distributions are available on the whole
real line, which are easy to use for data analysis purpose. The main
role of the extra parameters $\alpha$ and $\beta$ is that the BGN
distribution can be used to model skewed real data,
a feature which is very common in practice.
The BGN distribution with five parameters to control location, dispersion, modality and skewness has great flexibility.

As a second goal,
we emphasize the BGN distribution ability for
modeling bimodal phenomena,
which naturally occurs in the context of
image processing of SAR
imagery~\protect\cite{ElZaartZiou2007}.
Thus,
the BGN distribution is a candidate to SAR data modeling.
We sought a data analysis using
simulated and actual SAR imagery.
Also,
we compare the BGN model with several
existing models for SAR,
such as
the gamma distribution~\protect\cite{Delignonetal2002},
the $\mathcal{K}$ distribution~\protect\cite{Blacknell1994}, and
the $\mathcal{G}^0$ distribution~\protect\cite{freryetal1997a}.
The gamma distribution is regarded as standard model for the herein considered
types of SAR data~\protect\cite{Delignonetal2002}.
For such,
actual data is analyzed and
fitted according to all models above,
where
the original and corrected Akaike information and Bayesian information
criteria
are employed as
the goodness-of-fit
measures~\protect\cite{Gao2010}.

The rest of the paper is organized as follows.
In Section~\ref{section-distribution},
we define the BGN distribution,
derive its density,
and
discuss particular cases.
General expansions for the moments are derived in
Section~\ref{section-moments}.
Maximum likelihood estimation is investigated in Section~\ref{section-mle}.
In Section~\ref{section-rng},
we propose a random number generator for the new distribution.
Section~\ref{section-sar} details the SAR image analysis based on actual data.
Section~\ref{section-conclusion} gives some concluding remarks.

\section{The BGN Distribution}
\label{section-distribution}

The GN distribution with location parameter $\mu$, dispersion parameter $\sigma$
and shape parameter $s$ has probability density function (\mbox{pdf})
given by~\cite{nadarajah2005generalized}
$$
g(x)=
\frac{s}{2\sigma\,\Gamma(1/s)}
\exp\left\{-\left|\frac{x-\mu}{\sigma}\right|^s\right\},
\quad x\in \mathbb{R},
$$
where $\mu$ is a real number, and $\sigma$ and $s$ are positive real numbers.
For the special case $s=1$, the above density function reduces to the Laplace
distribution with location parameter $\mu$ and scale parameter $\sigma$.
Similarly,
for $s=2$, the normal distribution is obtained with mean $\mu$ and variance
$\sigma^2/2$. The main feature of the GN model is that new parameter $s$ can
introduce some skewness and kurtosis.

The GN cumulative function can be expressed
as follows~\cite[Eq.~5-6]{nadarajah2005generalized}
\begin{equation}\label{eq.F(x)}
G(x)=
\begin{cases}
\frac{\Gamma(1/s, (\frac{\mu-x}{\sigma})^s)}{2\Gamma(1/s)},
& x \leq \mu,\\
1-\frac{\Gamma(1/s, (\frac{x-\mu}{\sigma})^s)}{2\Gamma(1/s)},
& x > \mu,
\end{cases}
\end{equation}
where $\Gamma(s,x) = \int_x^\infty t^{s-1} \exp(-t) \mathrm{d}t$
is the complementary
incomplete gamma function.

The density function of the standardized random variable $Z=(X-\mu)/\sigma$
is
$$
\phi_s(z)=\frac{s}{2 \Gamma(1/s)}\exp(-|z|^s),\quad z\in \mathbb{R}.
$$
Thus, $g(x)=\sigma^{-1}\,\phi_s(\frac{x-\mu}{\sigma})$.
From (\ref{eq.F(x)}), the \mbox{cdf} of the standardized \text{GN}
distribution reduces to
\begin{align}
\Phi_s(z)=
\begin{cases}
\frac{\Gamma(1/s,(-z)^s)}{2\Gamma(1/s)},& z \leq 0,\\
1-\frac{\Gamma(1/s, z^s)}{2\Gamma(1/s)},& z > 0.
\end{cases}
\label{Phi}
\end{align}

Based on equation (\ref{cdfbetafamily}), we propose a natural generalization
of the GN distribution and provide a comprehensive treatment of its
mathematical properties. The \mbox{cdf} of the BGN distribution is
\begin{align*}
F(x)
=
\frac{1}{\mathrm{B}(\alpha,\beta)}
\int_0^{\Phi_s(\frac{x-\mu}{\sigma})}
t^{\alpha-1}\,(1-t)^{\beta-1}\,\mathrm{d}t
=
I_{\Phi_s(\frac{x-\mu}{\sigma})}(\alpha,\beta).
\label{cdfBGN}
\end{align*}
Consequently, its density function reduces to
\begin{equation}\label{bgndensity}
f(x)=\frac{1}{\sigma\mathrm{B}(\alpha,\beta)}\,
\left[\Phi_s\left(\frac{x-\mu}{\sigma}\right)\right]^{\alpha-1}
\left[1-\Phi_s\left(\frac{x-\mu}{\sigma}\right)
\right]^{\beta-1}\phi_s\left(\frac{x-\mu}{\sigma}\right).
\end{equation}
Here, the parameters $\alpha$ and $\beta$ control skewness through the
relative tail weights. They provide greater flexibility in the form of the
distribution and consequently in modeling observed real data.
A random variable $X$
with density function (\ref{bgndensity}) is denoted by
BGN$(\alpha,\beta,\mu,\sigma,s)$.
Clearly,
the beta normal and beta Laplace distributions are special models
of (\ref{bgndensity}) corresponding to $s=1$ and $s=2$, respectively.
The GN distribution is a special sub-model for $\alpha=\beta=1$.
The normal distribution
arises for $\alpha=\beta=1$ and $s=2$,
whereas the Laplace distribution corresponds
to $\alpha=\beta=s=1$. Location and scale parameters are
redundant in (\ref{bgndensity}),
since if $X\sim$BGN$(\alpha,\beta,0,1,s)$ then
$Y=\sigma\,X+\mu \sim$BGN$(\alpha,\beta,\mu,\sigma,s)$.

Figure~\ref{fig0} displays the BGN density function (\ref{bgndensity})
for selected parameter values. As parameters vary, several useful features,
such as bi-modality and pronounced skewness, can be obtained. These facts
illustrate the flexibility of the BGN distribution to analyze real data.
\begin{figure}
\centering
\subfigure[$\alpha=\beta=1$]{\epsfig{file=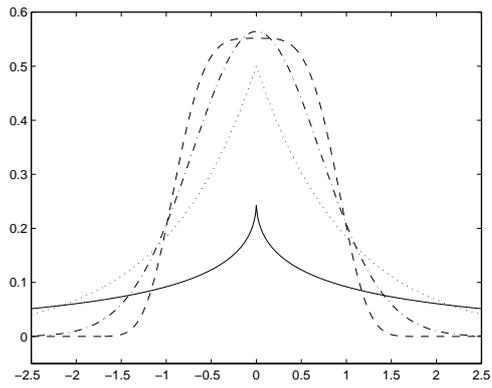,width=0.48\linewidth}}
\subfigure[$\alpha=\beta=0.5$]{\epsfig{file=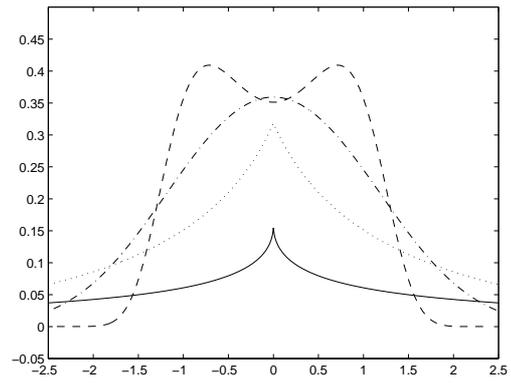,width=0.48\linewidth}}
\\
\subfigure[$\alpha=0.5, \beta=2$]{\epsfig{file=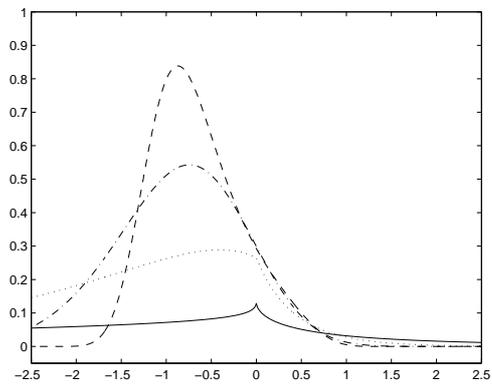,width=0.48\linewidth}}
\subfigure[$\alpha=0.25, \beta=0.5$]{\epsfig{file=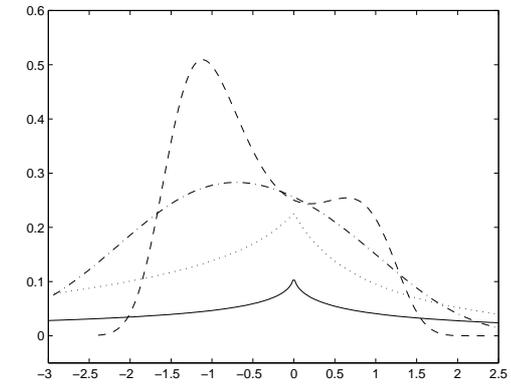,width=0.48\linewidth}}
\caption{Plots of the BGN density function for some parameter values:
$\mu=0$, $\sigma = 1$, and $s\in\{1/2,1,2,4\}$
(solid, dotted, dash-dotted, and dashed, respectively).}
\label{fig0}
\end{figure}

Note that the limiting behavior of the GN density with respect to~$s$
is given by
\begin{align*}
\lim_{s\to\infty}
\phi_s(z)=
\lim_{s\to\infty}
\frac{s}{2 \Gamma(1/s)}
\exp
(-|z|^s)=
\begin{cases}
\frac{1}{2}, & |z|<1, \\
0, & \text{otherwise.}
\end{cases}
\end{align*}

Consequently, the dominate convergence theorem leads to
\begin{align*}
\lim_{s\to\infty}
\Phi_s(z)
=
\begin{cases}
0, & z<-1, \\
\frac{1}{2}(z+1), & |z|<1,\\
1, &z>1.
\end{cases}
\end{align*}
Making appropriate substitutions, we conclude that the limiting
distribution $\lim_{s\to\infty}f(x)$ is related to the usual beta distribution
\begin{align}
\label{limit.beta}
\lim_{s\to\infty}
f(x)
&=
\frac{1}{2\sigma}
f_\text{Beta}
\left(
\frac{1}{2}\left( \frac{x-\mu}{\sigma}+1 \right)
\right),
\quad
x\in[\mu-\sigma,\mu+\sigma],
\end{align}
where
$f_\text{Beta}(x)=\frac{1}{\mathrm{B}(\alpha,\beta)}x^{\alpha-1}\,(1-x)^{\beta-1}$
is the beta density defined over $x\in[0,1]$.
Figure~\ref{fig1} illustrates this behavior.
As $s$ increases, plots in Figures~\ref{fig1}(a) and~\ref{fig1}(b) tend to the
U-shaped and triangular density functions, respectively,
which are special models of the beta distribution.
Figure~\ref{fig.diag} depicts a diagram with the relations among the discussed models.

\begin{figure}
\centering
\subfigure[$\alpha=\beta=0.5$]{\epsfig{file=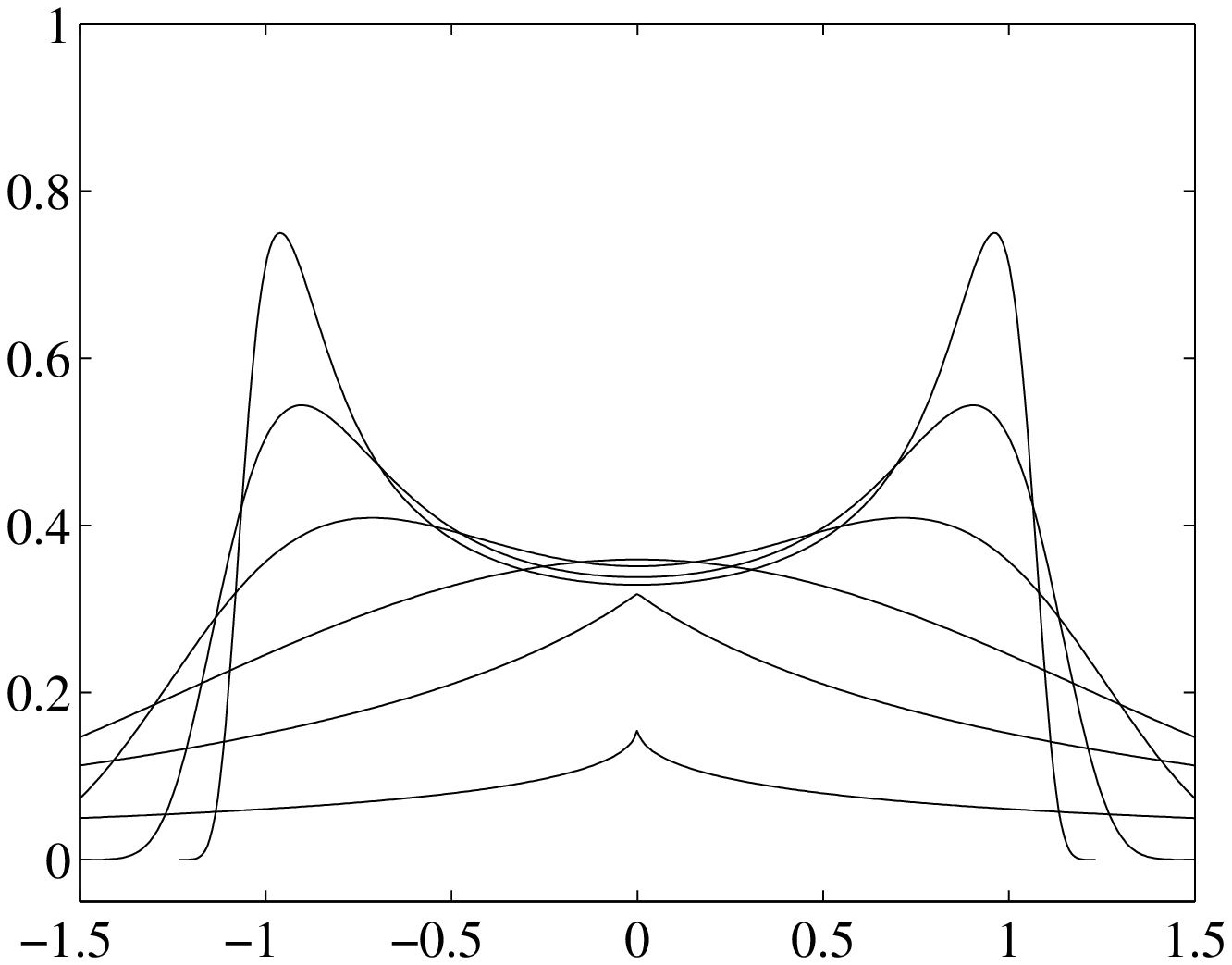,width=0.48\linewidth}}
\subfigure[$\alpha=2, \beta=1$]{\epsfig{file=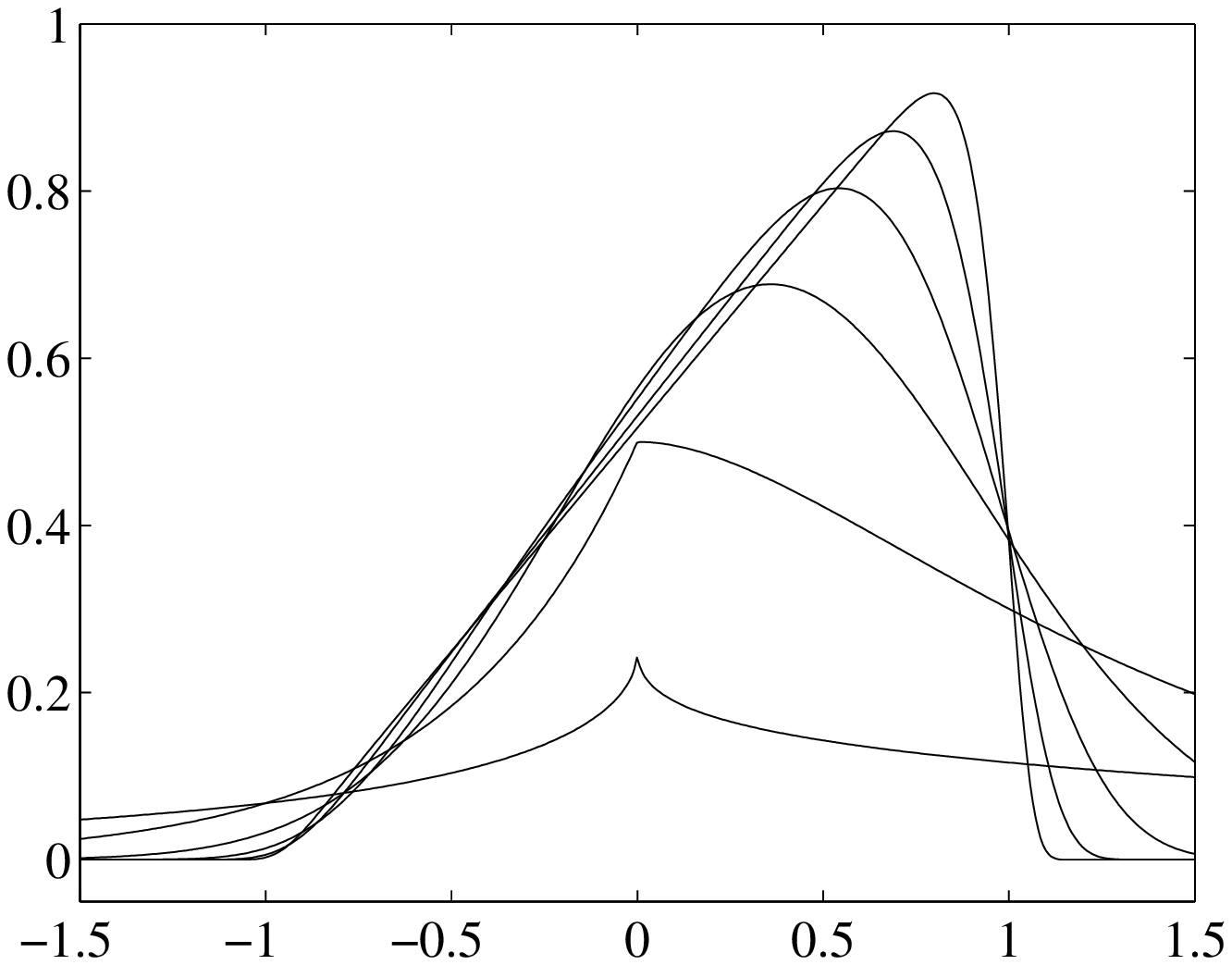,width=0.48\linewidth}}
\caption{Limiting behavior of the BGN density function for increasing values of $s\in\{1/2,1,2,4,8,16\}$,
$\mu=0$, and $\sigma=1$:
(a) U-shaped beta density and (b) triangular density.}
\label{fig1}
\end{figure}

\begin{figure}
\centering
\input{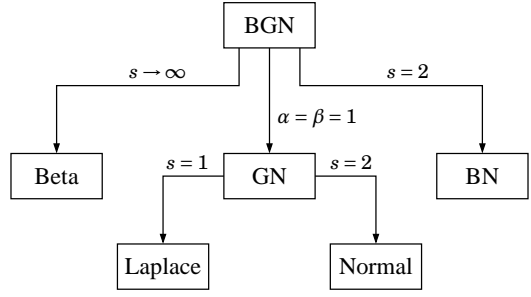}
\caption{Diagram relating the current models.}
\label{fig.diag}
\end{figure}

\section{Moments}
\label{section-moments}

Let $X$ be a random variable having the BGN$(\alpha,\beta,\mu,\sigma,s)$
distribution. The $n$th moment of $X$ becomes
\begin{align*}
\mathrm{E}(X^n)=
\frac{1}{\mathrm{B}(\alpha,\beta)}\frac{1}{\sigma}\int_{-\infty}^\infty x^n
\left[\Phi_s\left(\frac{x-\mu}{\sigma}\right)\right]^{\alpha-1}
\left[1-\Phi_s\left(\frac{x-\mu}{\sigma}\right)\right]^{\beta-1}
\phi_s \left(\frac{x-\mu}{\sigma}\right) \mathrm{d}x.
\end{align*}

\begin{proposition}
\label{proposition-moment}
The $n$th moment of $X$ can be expressed as
\begin{equation}
\label{eq.EXn}
\begin{split}
\mathrm{E}(X^n)
=
&
\frac{\mu^n}{\mathrm{B}(\alpha,\beta)}
\sum_{j=0}^{\infty}
(-1)^j\binom{\beta-1}{j}
\sum_{i=0}^n
\binom{n}{i}
\left(\frac{\sigma}{\mu}
\right)^i
\\
&
\times
\sum_{k=0}^\infty
\left[
v_k(j+\alpha-1)
+(-1)^{i+k}
\binom{j+\alpha-1}{k}
\right]
J_{i,k}^{(s)}
,
\end{split}
\end{equation}
where
\begin{align*}
v_k(\alpha)=\sum_{m=k}^\infty (-1)^{k+m}\,\binom{\alpha}{m}\,\binom{m}{k}
\end{align*}
and
\begin{align*}
J_{i,k}^{(s)}=\int_0^\infty z^i\,\phi_s(z)\,\Phi_s(z)^k \mathrm{d}z
.
\end{align*}
\end{proposition}

The above proposition is proved in Appendix~\ref{appendix-proof-proposition}.
It is also clear that the evaluation of the $n$th moment
is related to the computation of $J_{i,k}^{(s)}$.
Appendix~\ref{appendix-J}
provides an expansion derivation for $J_{i,k}^{(s)}$,
as shown in~\eqref{J-form-1}.

\section{Maximum Likelihood Estimation}
\label{section-mle}

Consider a random variable having the BGN distribution and
let $\mathbf{\theta}=(s,\mu,\sigma,\alpha,\beta)^T$ be its parameter
vector. The log-likelihood for a particular observation is
\begin{align*}
\ell(\mathbf{\theta})=&-
\log \mathrm{B}(\alpha,\beta)-
\log\sigma+
(\alpha-1)
\log
\Phi_s\left(\frac{x-\mu}{\sigma}\right)
\\&+(\beta-1)\log\left[1-\Phi_s\left(\frac{x-\mu}{\sigma}\right)\right]
+\log\phi_s\left(\frac{x-\mu}{\sigma}\right).
\end{align*}

To evaluate the unit score vector
$\left(
\frac{\mathrm{d}}{\mathrm{d}s}\ell,
\frac{\mathrm{d}}{\mathrm{d}\mu}\ell,
\frac{\mathrm{d}}{\mathrm{d}\sigma}\ell,
\frac{\mathrm{d}}{\mathrm{d}\alpha}\ell,
\frac{\mathrm{d}}{\mathrm{d}\beta}\ell
\right)^T$,
simple algebra yields
\begin{align*}
\frac{\mathrm{d}}{\mathrm{d}s}\ell
\left[\frac{\alpha-1}{\Phi_s\left(\frac{x-\mu}{\sigma}\right)}
-\frac{\beta-1}{1-\Phi_s\left(\frac{x-\mu}{\sigma}\right)}\right]
\frac{\mathrm{d}}{\mathrm{d}s}
\Phi_s\left(\frac{x-\mu}{\sigma}\right)
+
\frac{\mathrm{d}}{\mathrm{d}s}
\log
\phi_s\left(\frac{x-\mu}{\sigma}\right).
\end{align*}
From the definition of $\Phi_s$ given in (\ref{Phi}), it is clear that
$\frac{\mathrm{d}}{\mathrm{d}s}\Phi_s\left(\frac{x-\mu}{\sigma}\right)$
depends on the derivatives of the gamma and the incomplete gamma functions.
Then,
these particular derivatives deserve close examination as follows.
For $x>0$, the following derivative holds~\cite{abramowitz1965handbook,gueddes1990evaluation}:
\begin{align*}
\frac{\mathrm{d}}{\mathrm{d}s}
\Gamma(1/s,x^s)
&=
-
\frac{1}{s^2}
\Big[\log (x^s) \Gamma(1/s,x^s)+x^s T(3,1/s,x^s)
\Big]
-x\exp(-x^s)\log(x)\\
&\triangleq\tilde\psi(s,x),
\end{align*}
where the quantity $\tilde\psi(s,x)$ is an auxiliary function
for the sake of convenience of manipulation and
$T(3,\cdot,\cdot)$ is a particular case of the Meijer $G$-function
given by~\cite[p.~156]{gueddes1990evaluation}:
\begin{align*}
T(3,a,x)
=&
G_{2}^3\,{}_3^0
\left(
x
\left|
\begin{matrix}
0,0 \\
-1,a-1,-1
\end{matrix}\right.
\right)
\\
=&
\frac{x^{a-1}{}_2\mathrm{F}_2\left(a,a; 1+a,1+a; -x\right)}{a^2}
-
\frac{\Gamma(a)\log(x)}{x}
+
\frac{\Gamma(a)\psi(a)}{x}
,
\end{align*}
where $\psi(\cdot)$ is the digamma function.
Consequently, after some manipulations, we obtain
\begin{align*}
\frac{\mathrm{d}}{\mathrm{d}s}
\Phi_s\left(\frac{x-\mu}{\sigma}\right)
&=
\frac{1}{2\Gamma(1/s)}
\begin{cases}
s^2\tilde\psi\left( s, (\frac{\mu-x}{\sigma})^s\right)+ \psi(1/s)\Gamma(1/s, (\frac{\mu-x}{\sigma})^s),
&
x\leq\mu,
\\
-[s^2\tilde\psi\left( s, (\frac{x-\mu}{\sigma})^s\right)+ \psi(1/s)\Gamma(1/s, (\frac{x-\mu}{\sigma})^s)],
&
x>\mu.
\end{cases}
\end{align*}
Another required quantity is the following:
\begin{align*}
\log
\phi_s\left( \frac{x-\mu}{\sigma} \right)
&=
\log
\left[\frac{s}{2\Gamma(1/s)}
\exp\left( -\left|\frac{x-\mu}{\sigma} \right|^s \right)
\right]
\\
&=
\log s
-
\log(2\Gamma(1/s))
-
\left| \frac{x-\mu}{\sigma} \right|^s
.
\end{align*}
Therefore, we have
\begin{align*}
\frac{\mathrm{d}}{\mathrm{d}s}
\log
\phi_s\left(\frac{x-\mu}{\sigma}\right)
&=
\frac{1}{s}
+
\frac{\psi(1/s)}{s^2}
-
\left|\frac{x-\mu}{\sigma} \right|^s
\log
\left(
\left|\frac{x-\mu}{\sigma} \right|
\right).
\end{align*}
Now, we examine the derivatives of $\ell$ with respect to $\mu$ and $\sigma$.
They are given by
\begin{align*}
\frac{\mathrm{d}}{\mathrm{d}\mu}
\ell
=&
-
\frac{\alpha-1}{\sigma}
\frac{\phi_s\left(\frac{x-\mu}{\sigma}\right)}{\Phi_s\left(\frac{x-\mu}{\sigma}\right)}
+
\frac{\beta-1}{\sigma}
\frac{\phi_s\left(\frac{x-\mu}{\sigma}\right)}{1-\Phi_s\left(\frac{x-\mu}{\sigma}\right)}
+
\frac{s}{\sigma}
\left|\frac{x-\mu}{\sigma}\right|^{s-1}
\mathrm{sign}(x-\mu)
\end{align*}
and
\begin{align*}
\frac{\mathrm{d}}{\mathrm{d}\sigma}
\ell
=&
-
\frac{1}{\sigma}
-
\frac{\alpha-1}{\sigma^2}
(x-\mu)
\frac{\phi_s\left(\frac{x-\mu}{\sigma}\right)}{\Phi_s\left(\frac{x-\mu}{\sigma}\right)}
+
\frac{\beta-1}{\sigma^2}
(x-\mu)
\frac{\phi_s\left(\frac{x-\mu}{\sigma}\right)}{1-\Phi_s\left(\frac{x-\mu}{\sigma}\right)}
-
\frac{s}{\sigma}
\left|\frac{x-\mu}{\sigma}\right|^s,
\end{align*}
respectively, where we use the fact that $|x| = x \,\mathrm{sign}(x)$.

In order to calculate the derivative with respect to $\alpha$, note that:
\begin{align*}
\frac{\mathrm{d}}{\mathrm{d}\alpha}
\log\mathrm{B}(\alpha,\beta)
&=
\psi(\alpha) - \psi(\alpha+\beta).
\end{align*}
Therefore,
it follows immediately that
\begin{align*}
\frac{\mathrm{d}}{\mathrm{d}\alpha}
\ell
=&
\psi(\alpha+\beta) - \psi(\alpha) +
\log\Phi_s\left(\frac{x-\mu}{\sigma}\right)
\end{align*}
and
\begin{align*}
\frac{\mathrm{d}}{\mathrm{d}\beta}
\ell
=&
\psi(\alpha+\beta) - \psi(\beta) +
\log\left[1-\Phi_s\left(\frac{x-\mu}{\sigma}\right)\right].
\end{align*}

\section{BGN Random Number Generator}
\label{section-rng}

In this section,
a random number generator (RNG) for the BGN distribution
is introduced.
This RNG allows a Monte Carlo study
to determine
the influence of the BGN shape parameters $\alpha$, $\beta$,  and $s$.

We consider the inverse transform algorithm
for generating continuous random variables~\protect{\cite[p.~67]{ross2006simulation}}.
Indeed,
a BGN random variable $X\sim \text{BGN}(\alpha,\beta,\mu,\sigma,s)$
can be related to
a beta distributed random variable $U\sim \operatorname{B}(\alpha,\beta)$
according to
$X = G^{-1}(U)$,
where $G$ is given in~\eqref{eq.F(x)}.
Notice that
\begin{align*}
G(x)
=
\begin{cases}
\frac{1}{2}
\left\{
1 - F_\Gamma\left[\left(-\frac{x-\mu}{\sigma}\right)^s\right]
\right\}
,
&x\leq\mu, \\
\frac{1}{2}
\left\{
1 + F_\Gamma\left[\left(\frac{x-\mu}{\sigma}\right)^s\right]
\right\}
,
&x>\mu,
\end{cases}
\end{align*}
where $F_\Gamma(\cdot)$ is the cdf of the gamma distribution
with shape and scape parameters given by $1/s$ and 1,
respectively.

Let $x$ and $u$ be realizations of $X$ and $U$, respectively.
For $x\leq\mu$,
we obtain the inverse transformation according to the following
manipulation
\begin{align*}
\frac{1}{2}
\left\{
1 - F_\Gamma\left[\left(-\frac{x-\mu}{\sigma}\right)^s\right]
\right\}
=
u.
\end{align*}
Therefore,
\begin{align*}
x
=
\mu
-
\sigma
\left[F^{-1}_\Gamma(1-2u)\right]^{1/s}
,
\end{align*}
where $F^{-1}_\Gamma(\cdot)$ is the quantile function of the
gamma distribution with parameters $1/s$ and 1.
Since
$0\leq F_\Gamma\left[\left(-\frac{x-\mu}{\sigma}\right)^s\right]\leq1$, we have $0 \leq u \leq 1/2$.

Analogously,
for $x>\mu$,
we obtain
\begin{align*}
x
=
\mu
+
\sigma
\left[
F^{-1}_\Gamma(2u-1)
\right]^{1/s}
,
\quad
1/2 < u \leq 1
.
\end{align*}

Therefore,
the RNG for $X$ can be algorithmically described as follows:
\begin{algorithmic}[1]
\State Generate $U\sim \text{beta}(\alpha,\beta)$
\If{$U\in[0,1/2]$}
   \State $X=\mu-\sigma \left[F^{-1}_\Gamma(1-2u)\right]^{1/s}$
\Else
   \State $X=\mu+\sigma \left[F^{-1}_\Gamma(2u-1)\right]^{1/s}$
\EndIf
\State \textbf{return} $X$.
\end{algorithmic}

Figure~\ref{passa0} displays
four cases where theoretical curves are compared with
randomly generated points for
$\alpha=0.1$,
$\beta=0.3$,
$\mu=0$,
$\sigma=1$, and
$s\in\{2,4,6,12\}$.

\begin{figure}
\centering
\psfrag{BGN1}[c][c][1.0]{\;\;\;\;\;\scriptsize{Generated values}}
\psfrag{BGN2}[c][c][1.0]{\;\;\;\;\;\scriptsize{Generated values}}
\psfrag{BGN3}[c][c][1.0]{\;\;\;\;\;\scriptsize{Generated values}}
\psfrag{BGN4}[c][c][1.0]{\;\;\;\;\;\scriptsize{Generated values}}
\psfrag{Density}[c][c][1.0]{\;\;\;\;\;\scriptsize{Density}}

\subfigure[BGN($0.1,0.3,0,1,2$)\label{powerr1}]{\includegraphics[width=.48\linewidth]{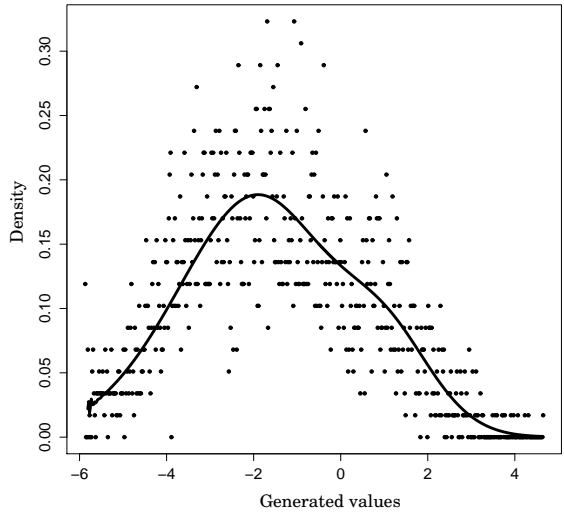}}
\subfigure[BGN($0.1,0.3,0,1,4$)\label{powerr2}]{\includegraphics[width=.48\linewidth]{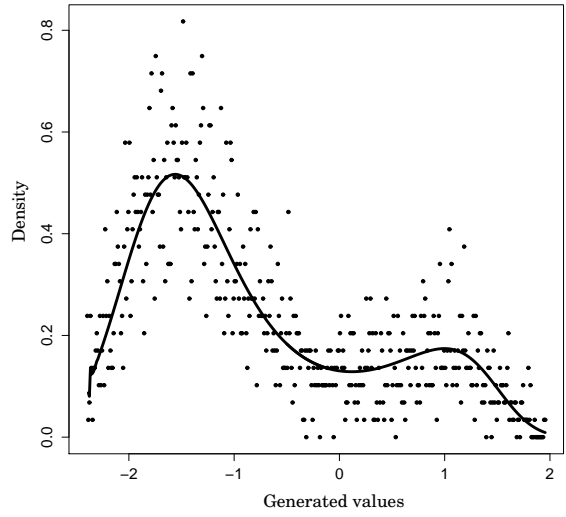}}
\subfigure[BGN($0.1,0.3,0,1,6$)\label{powerr3}]{\includegraphics[width=.48\linewidth]{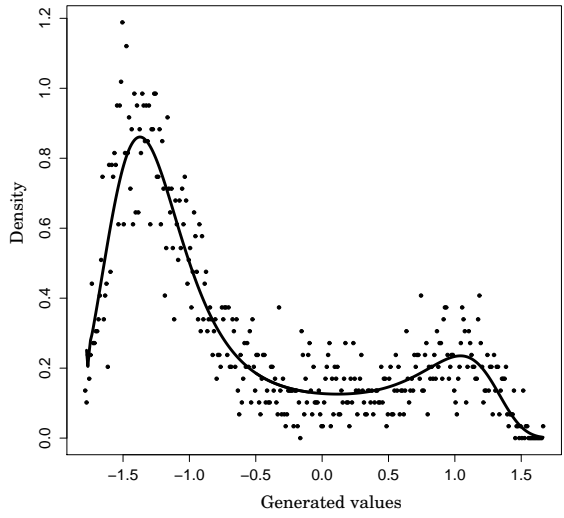}}
\subfigure[BGN($0.1,0.3,0,1,12$)\label{powerr4}]{\includegraphics[width=.48\linewidth]{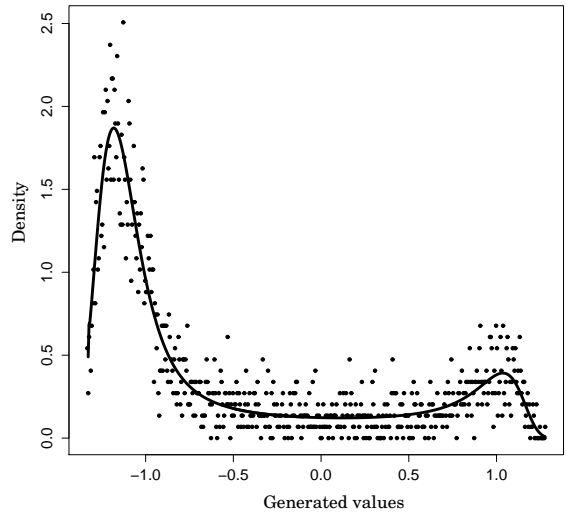}}
\caption{BGN random number generation.}
\label{passa0}
\end{figure}

\section{SAR Image Processing}
\label{section-sar}

In this section,
we assess the proposed statistical modeling by means
of simulated and actual data analysis.
First,
we employ the proposed BGN RNG to generate
synthetic data
to be submitted to ML estimation.
We show evidence of the effectiveness of the derived
methods.
Subsequently,
we employ ML estimation to actual SAR data
aiming their statistical modeling.

\subsection{Influence Study Based on Shape Parameters}

To assess the effect of the shape parameters,
we performed a Monte Carlo
simulation with $M=10^3$ replications.

For each replication
the following
steps were considered:

\begin{enumerate}[(i)]
\item
Simulated BGN distributed
images of $7\times 7$, $11 \times 11$, and $20 \times 20$ pixels
were obtained by means of the BGN RNG
furnishing sample sizes of $N\in\{49,121,400\}$.

\item
Three scenarios were considered:
(a)~$\alpha=1$, $\beta=1$, and $s\in\{1,1.5,2,1.5,3,3.5,4,4.5,5\}$;
(b)~$\alpha=1$, $\beta\in\{1,1.5,2,1.5,3,3.5,4,4.5,5\}$, and $s=1$;
and
(c)~$\alpha\in\{1,1.5,2,1.5,3,3.5,4,4.5,5\}$, $\beta=1$, and $s=1$.
In all cases,
$\mu=0$, and $\sigma=1$;

\item
Generated data was submitted to ML estimation
to obtain estimate parameters
and an estimated BGN pdf $\hat{f}(\cdot)$;

\item
Squared errors between the exact and estimated pdfs were computed.

\end{enumerate}

This procedure was repeated $M$ times, which furnished the
mean squared error (MSE).
Figure~\ref{influ} displays the relationship between shape parameters and the MSE.

As expected from the ML estimation asymptotic properties, in general terms, the influence on the MSE diminishes when the sample size increases.

\begin{figure}
\centering
\psfrag{g1}[c][c][1.0]{\;\;\;\;\;\scriptsize{$\hskip+8ex \beta\colon \alpha=s=1$}}
\psfrag{g2}[c][c][1.0]{\;\;\;\;\;\scriptsize{$\hskip+7.7ex \alpha \colon \beta=s=1$}}
\psfrag{g3}[c][c][1.0]{\;\;\;\;\;\scriptsize{$\hskip+8.3ex s \colon \alpha=\beta=1$}}
\psfrag{Influence}[c][c][1.0]{\;\;\;\;\;\scriptsize{MSE}}
\psfrag{ABA}[c][c][1.0]{\;\;\;\;\;\scriptsize{Parameter values}}

\subfigure[$N=49$\label{influ1}]{\includegraphics[width=.48\linewidth]{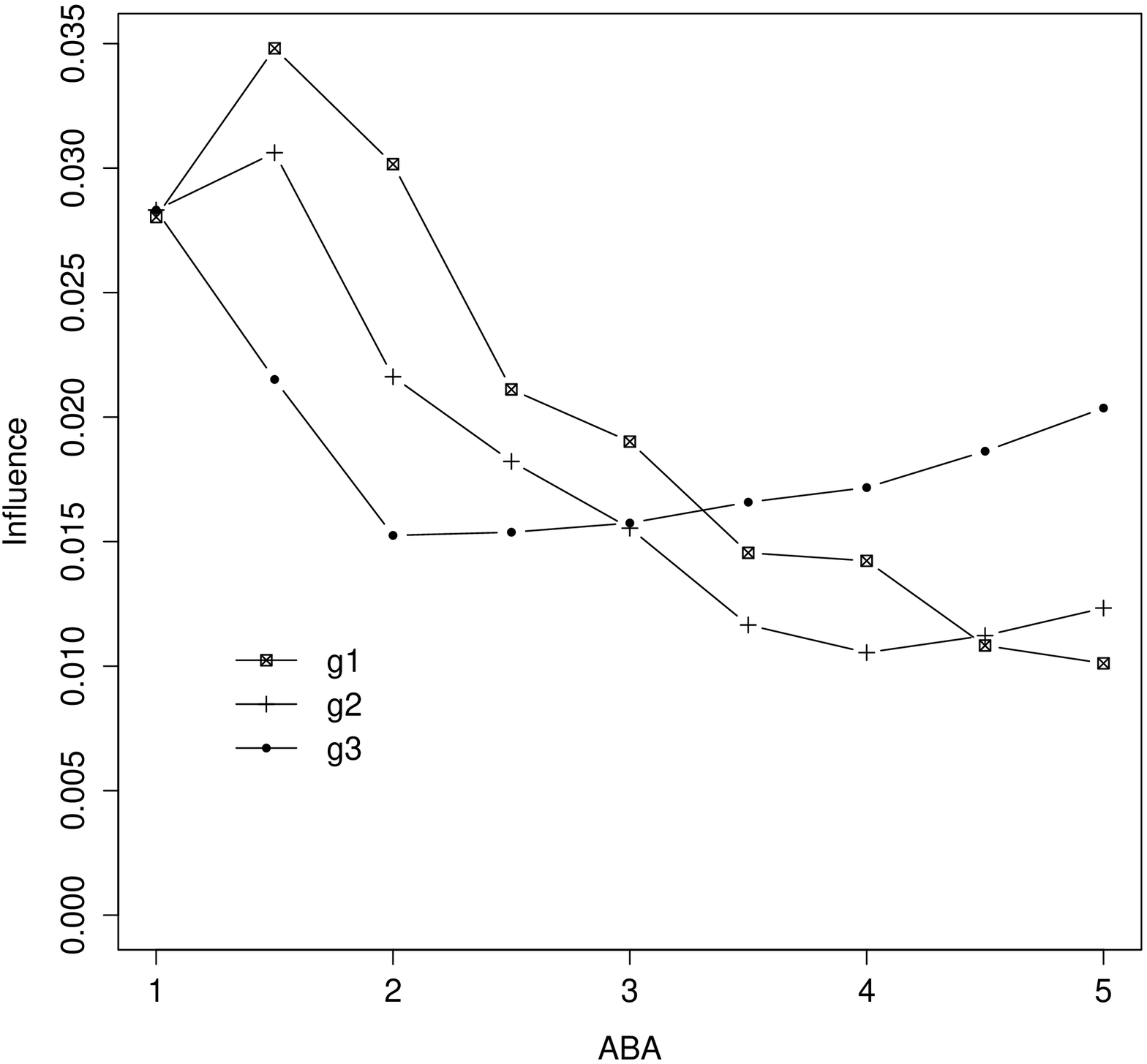}}
\subfigure[$N=121$\label{influ2}]{\includegraphics[width=.48\linewidth]{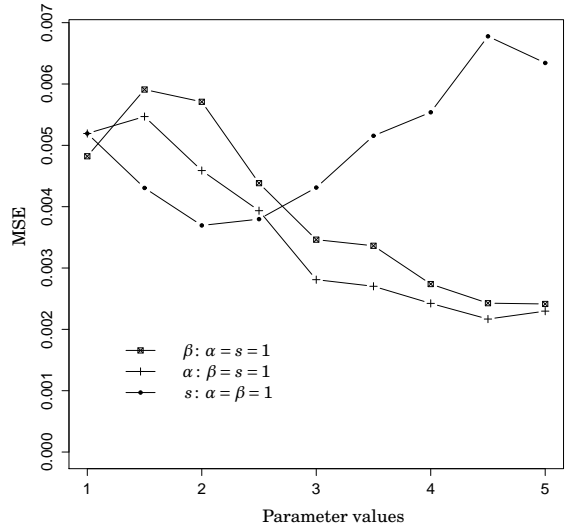}} \\
\subfigure[$N=400$\label{influ3}]{\includegraphics[width=.48\linewidth]{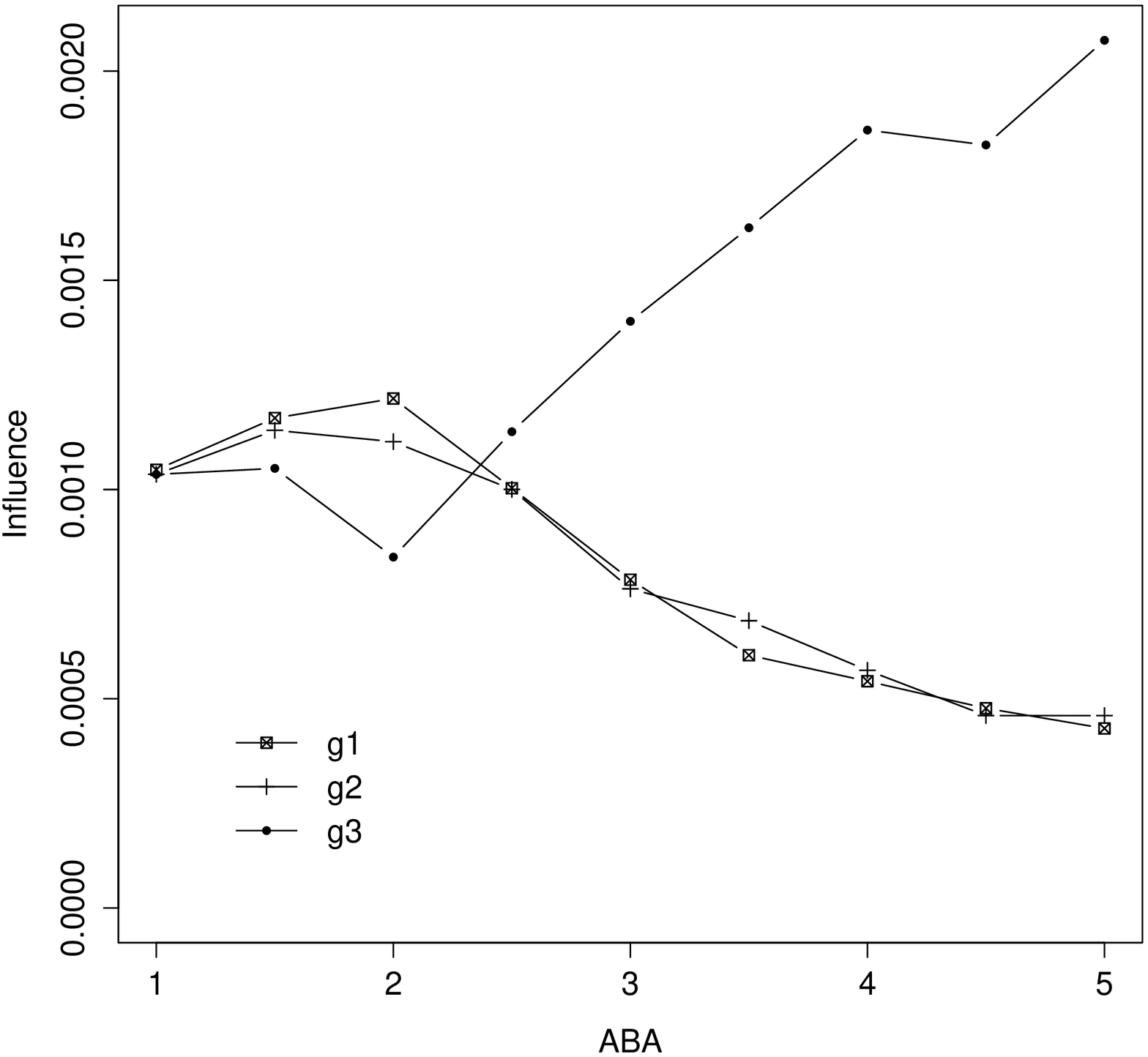}}
\caption{Mean square error for several pararameter values.}
\label{influ}
\end{figure}

\subsection{SAR Image Modeling}

In recent years, the interest in understanding such type of imagery in a multidimensional and multilook perspective has increased.
In this case,
the current data is termed multilook polarimetric SAR (PolSAR).
In such situation,
backscattered signals are recorded as complex elements for all possible combinations of linear reception and transmission polarizations: HH, HV, and VV (H for horizontal and V for vertical polarization).
In particular, the intensity of the echoed signal polarization channels plays an important role, since it depends on the physical properties of the target surface.
Figure~\ref{Foulum} presents an image over the surroundings of
Foulum (Denmark) obtained by the SAR system EMISAR~\cite{Doulgerisetal2011}.
This is a polarimetric SAR image,
i.e.,
their pixels are represented by
3$\times$3 Hermitian positive definite matrices
whose diagonal elements are positive real intensities:
denominated by HH, HV, and~VV.

\begin{figure}
\centering
\psfrag{A1}[c][c][1.0]{\footnotesize{\quad\quad\quad\quad\quad\quad\quad\color{white}\textbf A1}}
\psfrag{A2}[c][c][1.0]{\footnotesize{\quad\color{black}\textbf A2}}
\psfrag{A3}[c][c][1.0]{\footnotesize{\color{white}\textbf A3}}
\psfrag{A}[c][c][1.0]{\;\;\;\;\;\tiny{BGN}}
\psfrag{B}[c][c][1.0]{\;\;\;\;\;\;\;\;\;\tiny{GAMMA}}
\psfrag{YP}[c][c][1.0]{\tiny{Empirical Density}}
\psfrag{YP1}[c][c][1.0]{\tiny{Empirical Density}}
\psfrag{XP}[c][c][1.0]{\tiny{Returned Signal Intensity}}
\psfrag{XP1}[c][c][1.0]{\tiny{Returned Signal Intensity}}
\psfrag{Density}[c][c][1.0]{\;\;\;\;\;\footnotesize{Density}}
\psfrag{HH1}[c][c][1.0]{\;\;\;\;\;\footnotesize{Intensity ($\times$ 10$^{-2}$)}}
\psfrag{HH2}[c][c][1.0]{\;\;\;\;\;\footnotesize{Intensity ($\times$ 10$^{-1}$)}}
\psfrag{HH3}[c][c][1.0]{\;\;\;\;\;\footnotesize{Intensity ($\times$ 10$^{-3}$)}}
\psfrag{HV1}[c][c][1.0]{\;\;\;\;\;\footnotesize{Intensity ($\times$ 10$^{-3}$)}}
\psfrag{HV2}[c][c][1.0]{\;\;\;\;\;\footnotesize{Intensity ($\times$ 10$^{-2}$)}}
\psfrag{HV3}[c][c][1.0]{\;\;\;\;\;\footnotesize{Intensity ($\times$ 10$^{-4}$)}}
\psfrag{VV1}[c][c][1.0]{\;\;\;\;\;\footnotesize{Intensity ($\times$ 10$^{-2}$)}}
\psfrag{VV2}[c][c][1.0]{\;\;\;\;\;\footnotesize{Intensity ($\times$ 10$^{-2}$)}}
\psfrag{VV3}[c][c][1.0]{\;\;\;\;\;\footnotesize{Intensity ($\times$ 10$^{-3}$)}}
\psfrag{0}[c][c][1.0]{\tiny{0}}
\psfrag{500}[c][c][1.0]{\tiny{500}}
\psfrag{1000}[c][c][1.0]{\tiny{1000}}
\psfrag{1500}[c][c][1.0]{\tiny{1500}}
\psfrag{2000}[c][c][1.0]{\tiny{2000}}
\psfrag{4000}[c][c][1.0]{\tiny{4000}}
\psfrag{6000}[c][c][1.0]{\tiny{6000}}
\psfrag{8000}[c][c][1.0]{\tiny{8000}}

\psfrag{2e-04}[c][c][1.0]{\tiny{2}}
\psfrag{4e-04}[c][c][1.0]{\tiny{4}}
\psfrag{6e-04}[c][c][1.0]{\tiny{6}}
\psfrag{8e-04}[c][c][1.0]{\tiny{8}}
\psfrag{1e-03}[c][c][1.0]{\tiny{10}}

\psfrag{0.0010}[c][c][1.0]{\tiny{1.0}}
\psfrag{0.0015}[c][c][1.0]{\tiny{1.5}}
\psfrag{0.0020}[c][c][1.0]{\tiny{2.0}}
\psfrag{0.0025}[c][c][1.0]{\tiny{2.5}}
\psfrag{0.0030}[c][c][1.0]{\tiny{3.0}}
\psfrag{0.0035}[c][c][1.0]{\tiny{3.5}}

\psfrag{0.000}[c][c][1.0]{\tiny{0}}
\psfrag{0.001}[c][c][1.0]{\tiny{1}}
\psfrag{0.002}[c][c][1.0]{\tiny{2}}
\psfrag{0.003}[c][c][1.0]{\tiny{3}}
\psfrag{0.004}[c][c][1.0]{\tiny{4}}
\psfrag{0.005}[c][c][1.0]{\tiny{5}}
\psfrag{0.010}[c][c][1.0]{\tiny{10}}
\psfrag{0.015}[c][c][1.0]{\tiny{15}}
\psfrag{0.020}[c][c][1.0]{\tiny{20}}

\psfrag{1}[c][c][1.0]{\tiny{1}}
\psfrag{2}[c][c][1.0]{\tiny{2}}
\psfrag{3}[c][c][1.0]{\tiny{3}}
\psfrag{4}[c][c][1.0]{\tiny{4}}
\psfrag{5}[c][c][1.0]{\tiny{5}}
\psfrag{6}[c][c][1.0]{\tiny{6}}
\psfrag{8}[c][c][1.0]{\tiny{8}}
\psfrag{10}[c][c][1.0]{\tiny{10}}
\psfrag{12}[c][c][1.0]{\tiny{12}}
\psfrag{14}[c][c][1.0]{\tiny{14}}
\psfrag{15}[c][c][1.0]{\tiny{15}}
\psfrag{20}[c][c][1.0]{\tiny{20}}
\psfrag{30}[c][c][1.0]{\tiny{30}}
\psfrag{100}[c][c][1.0]{\tiny{100}}
\psfrag{200}[c][c][1.0]{\tiny{200}}
\psfrag{300}[c][c][1.0]{\tiny{300}}
\psfrag{400}[c][c][1.0]{\tiny{400}}

\psfrag{0.00}[c][c][1.0]{\tiny{0}}
\psfrag{0.05}[c][c][1.0]{\tiny{5}}
\psfrag{0.10}[c][c][1.0]{\tiny{10}}
\psfrag{0.15}[c][c][1.0]{\tiny{15}}
\psfrag{0.20}[c][c][1.0]{\tiny{20}}
\psfrag{0.25}[c][c][1.0]{\tiny{25}}
\psfrag{0.30}[c][c][1.0]{\tiny{30}}

\psfrag{0.00}[c][c][1.0]{\tiny{0}}
\psfrag{0.05}[c][c][1.0]{\tiny{5}}
\psfrag{0.10}[c][c][1.0]{\tiny{10}}
\psfrag{0.15}[c][c][1.0]{\tiny{15}}
\psfrag{0.20}[c][c][1.0]{\tiny{20}}
\psfrag{0.25}[c][c][1.0]{\tiny{25}}
\psfrag{0.30}[c][c][1.0]{\tiny{30}}

\psfrag{0.0}[c][c][1.0]{\tiny{0}}
\psfrag{0.2}[c][c][1.0]{\tiny{2}}
\psfrag{0.4}[c][c][1.0]{\tiny{4}}
\psfrag{0.6}[c][c][1.0]{\tiny{6}}
\psfrag{0.8}[c][c][1.0]{\tiny{8}}

\subfigure[Foulum (Denmark) \label{Foulum}]{\includegraphics[width=.4\linewidth]{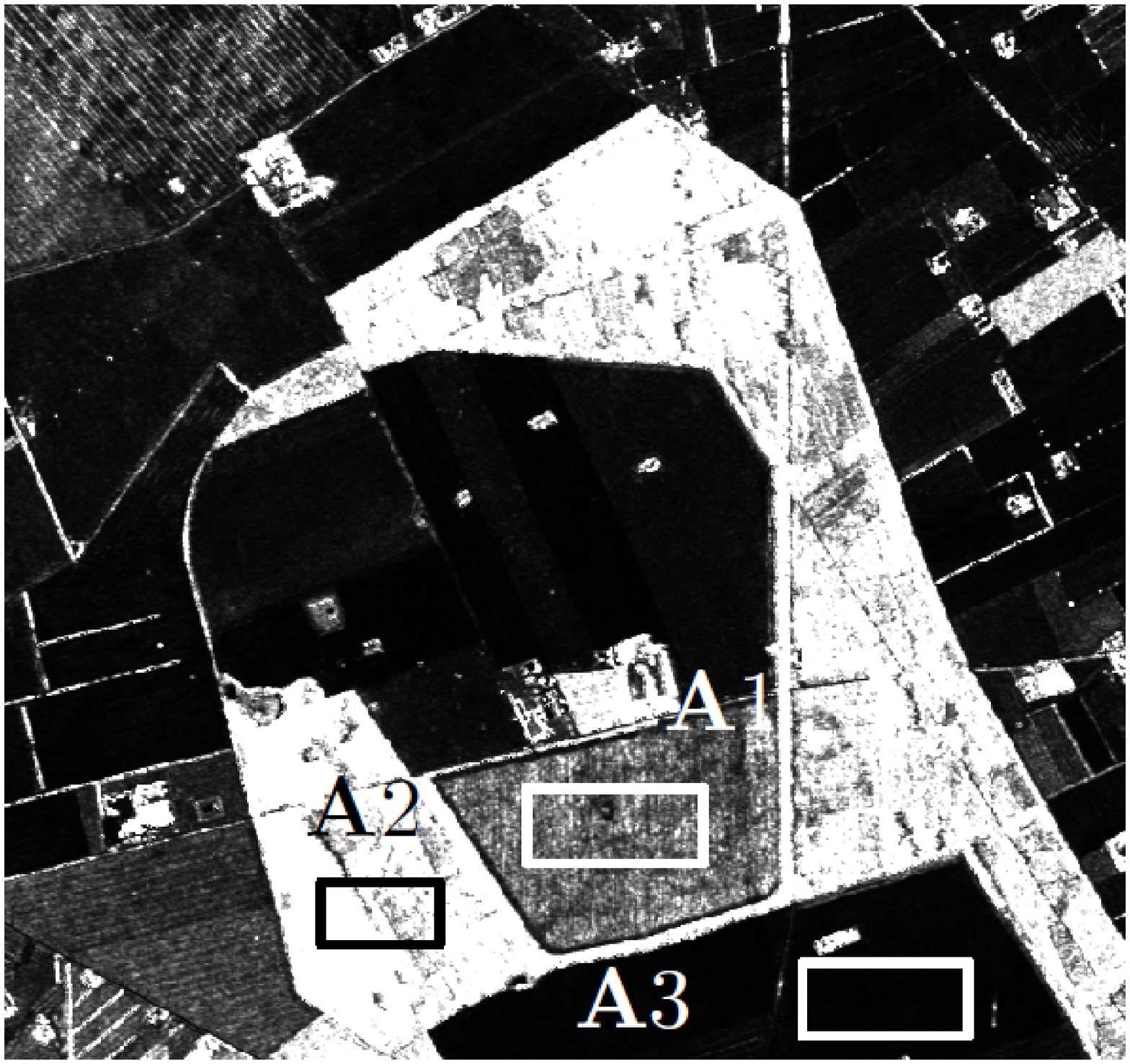}}\\

\subfigure[(channel,region)=(HH,A1) \label{a11}]{\includegraphics[width=.3\linewidth]{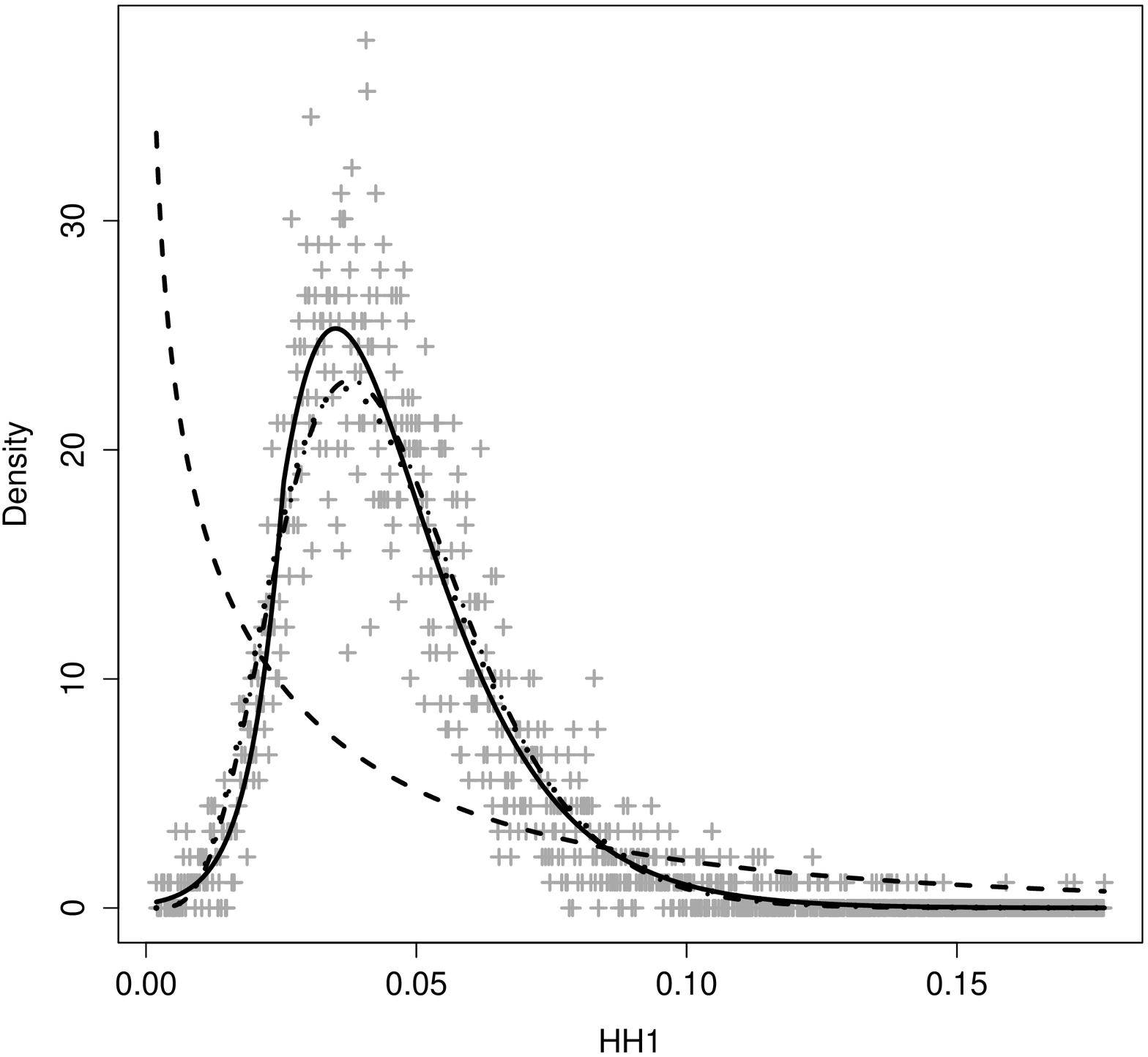}}
\subfigure[(channel,region)=(HV,A1)  \label{a12}]{\includegraphics[width=.3\linewidth]{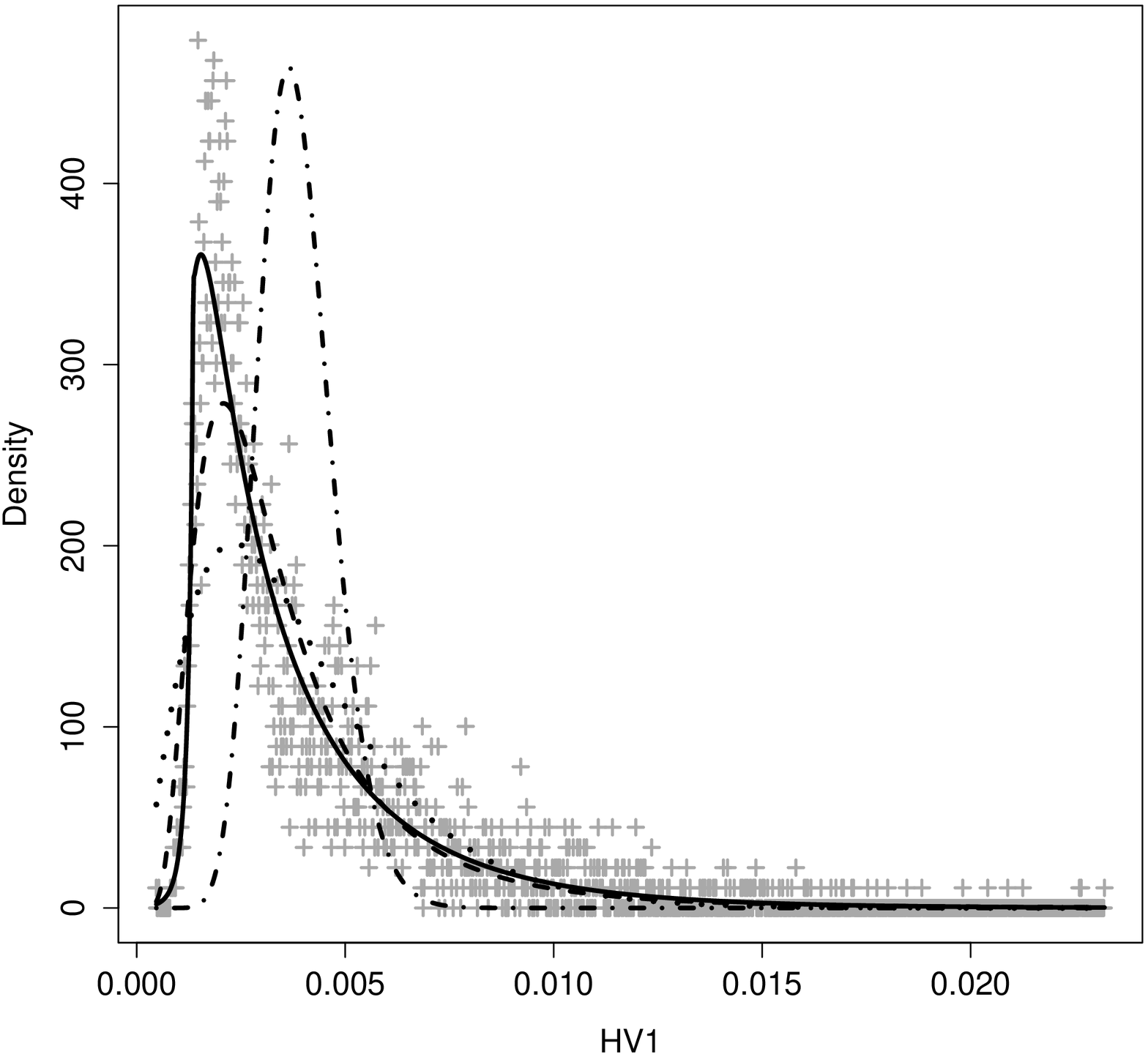}}
\subfigure[(channel,region)=(VV,A1)  \label{a13}]{\includegraphics[width=.3\linewidth]{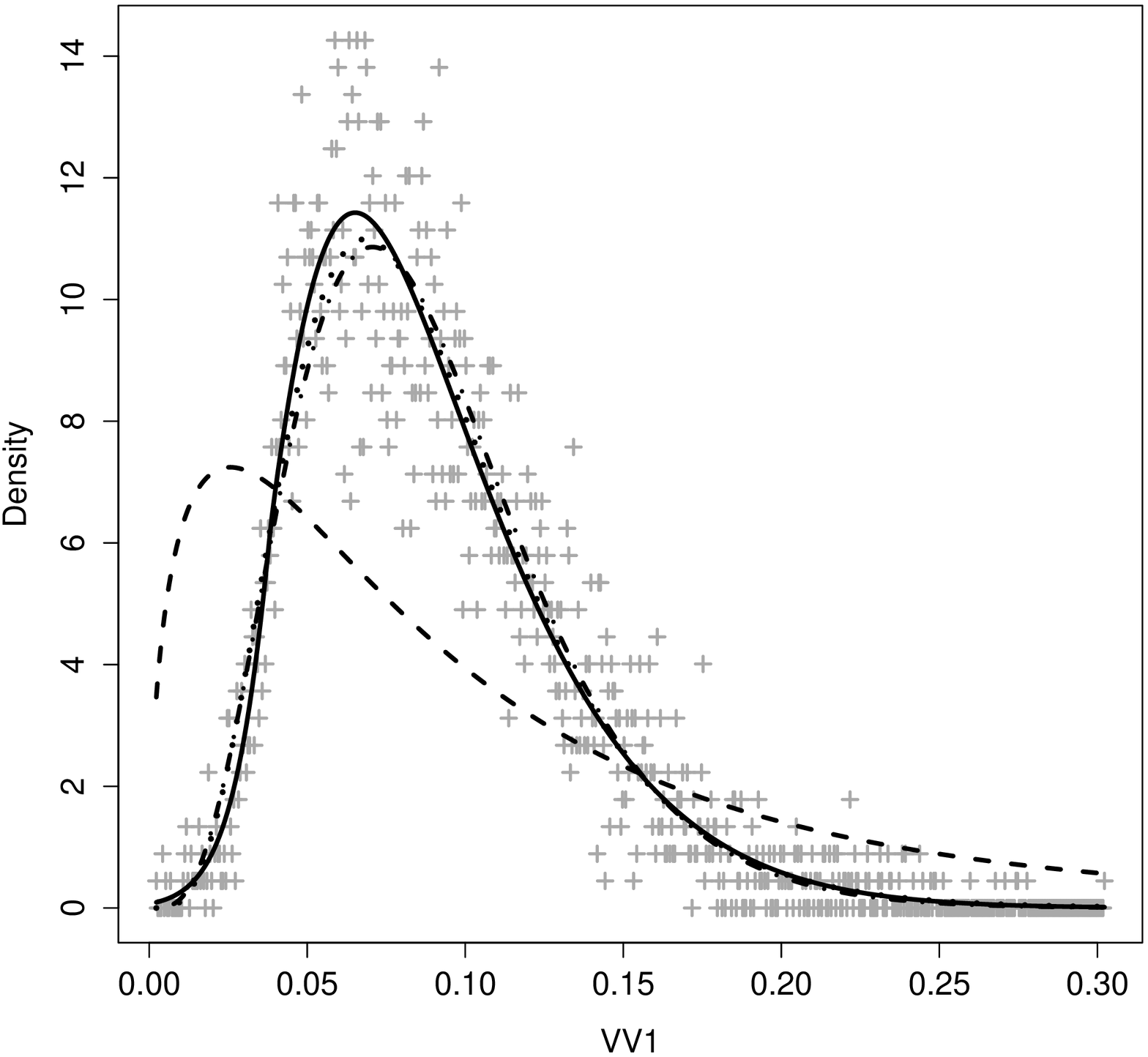}}\\

\subfigure[(channel,region)=(HH,A2) \label{a21}]{\includegraphics[width=.3\linewidth]{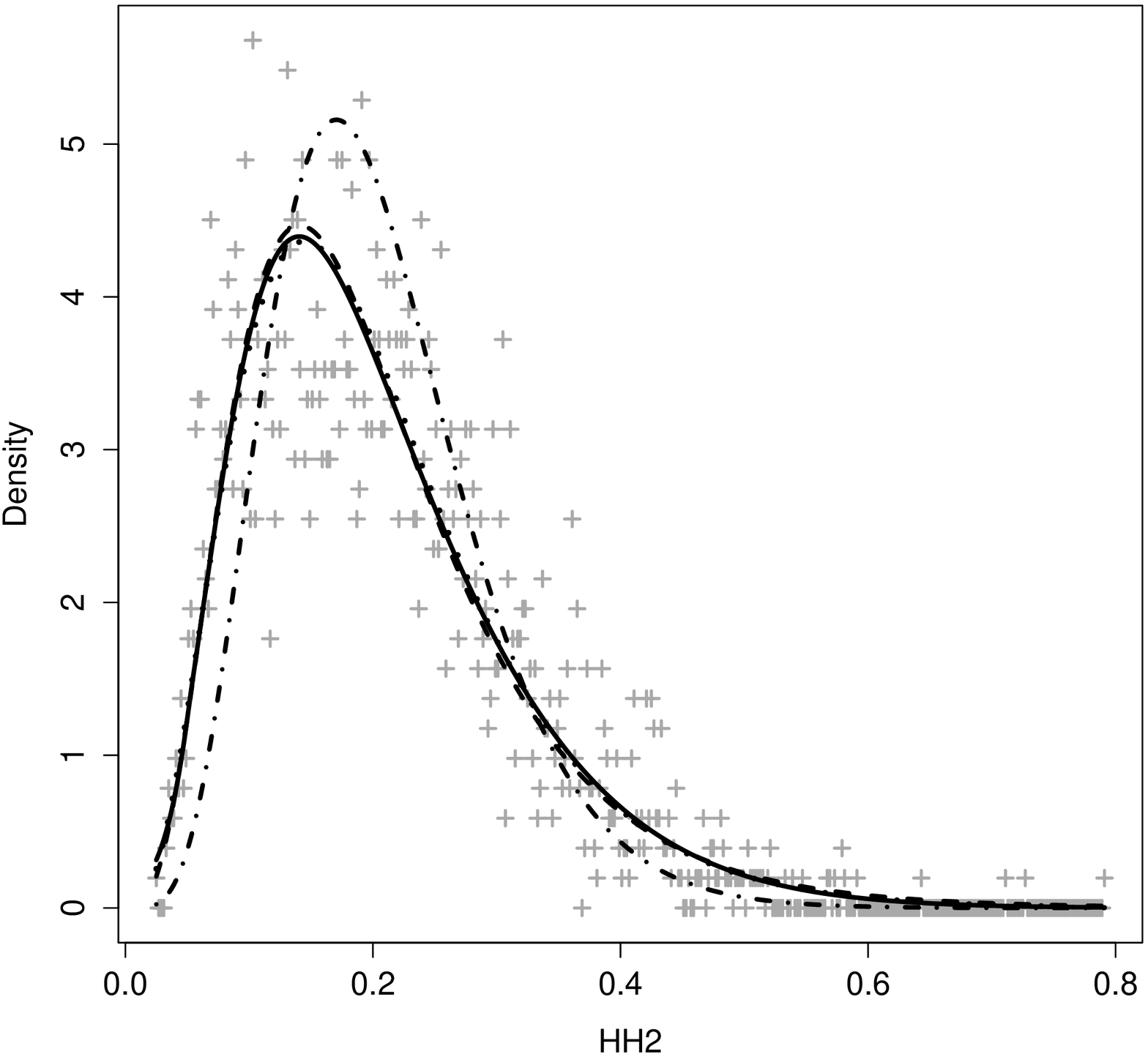}}
\subfigure[(channel,region)=(HV,A2)  \label{a22}]{\includegraphics[width=.3\linewidth]{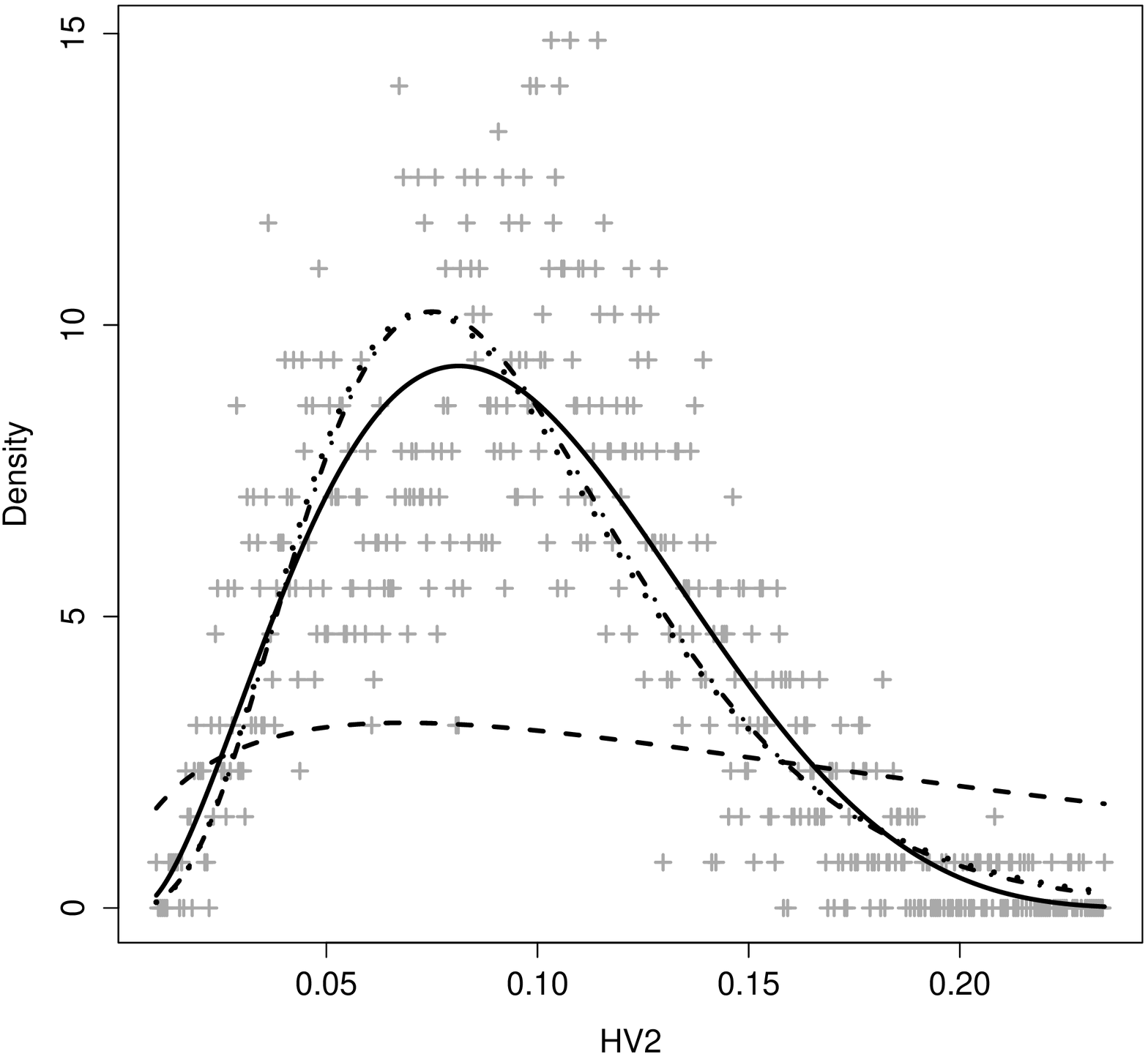}}
\subfigure[(channel,region)=(VV,A2) \label{a23}]{\includegraphics[width=.3\linewidth]{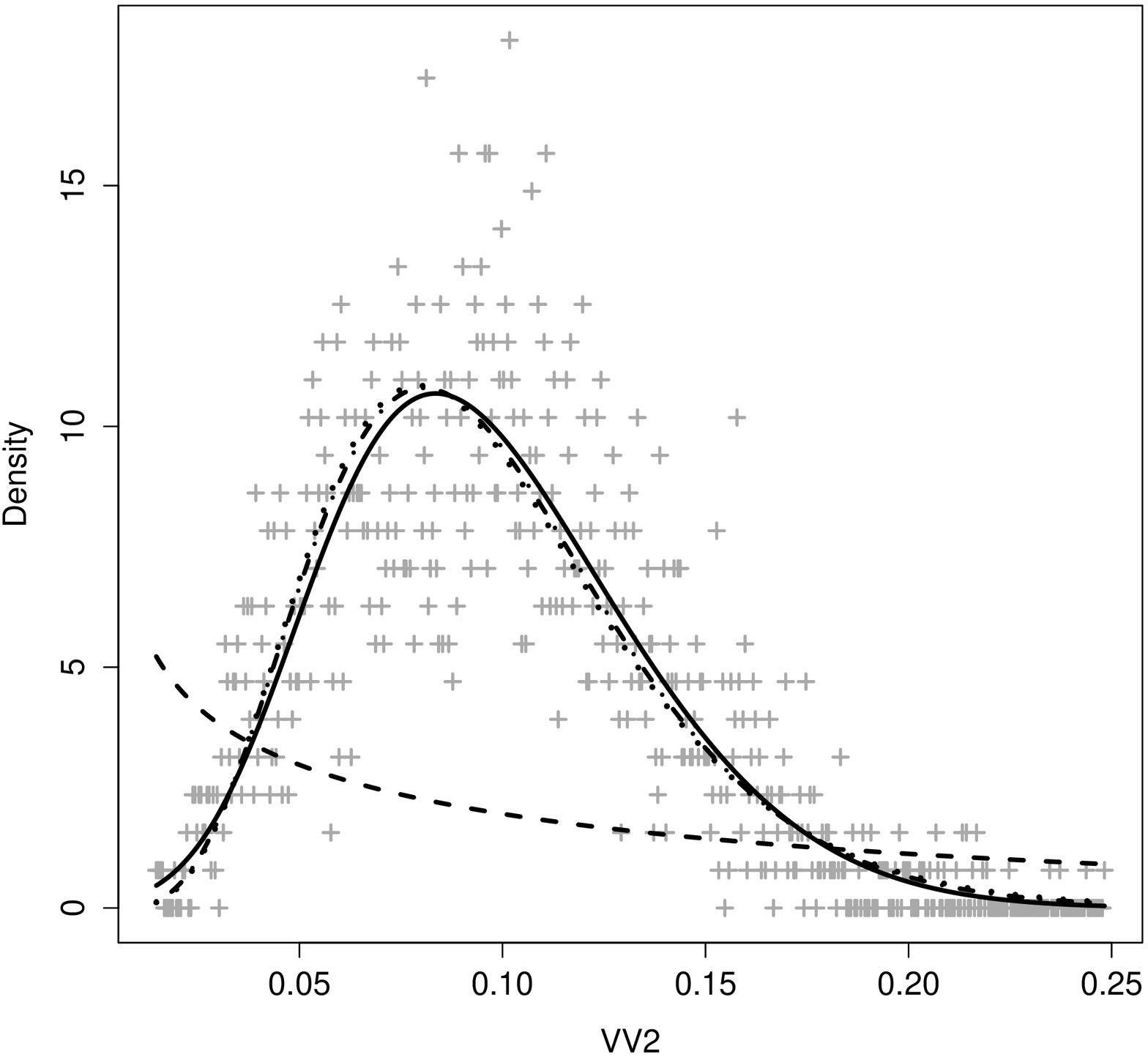}}\\

\subfigure[(channel,region)=(HH,A3)\label{a31}]{\includegraphics[width=.3\linewidth]{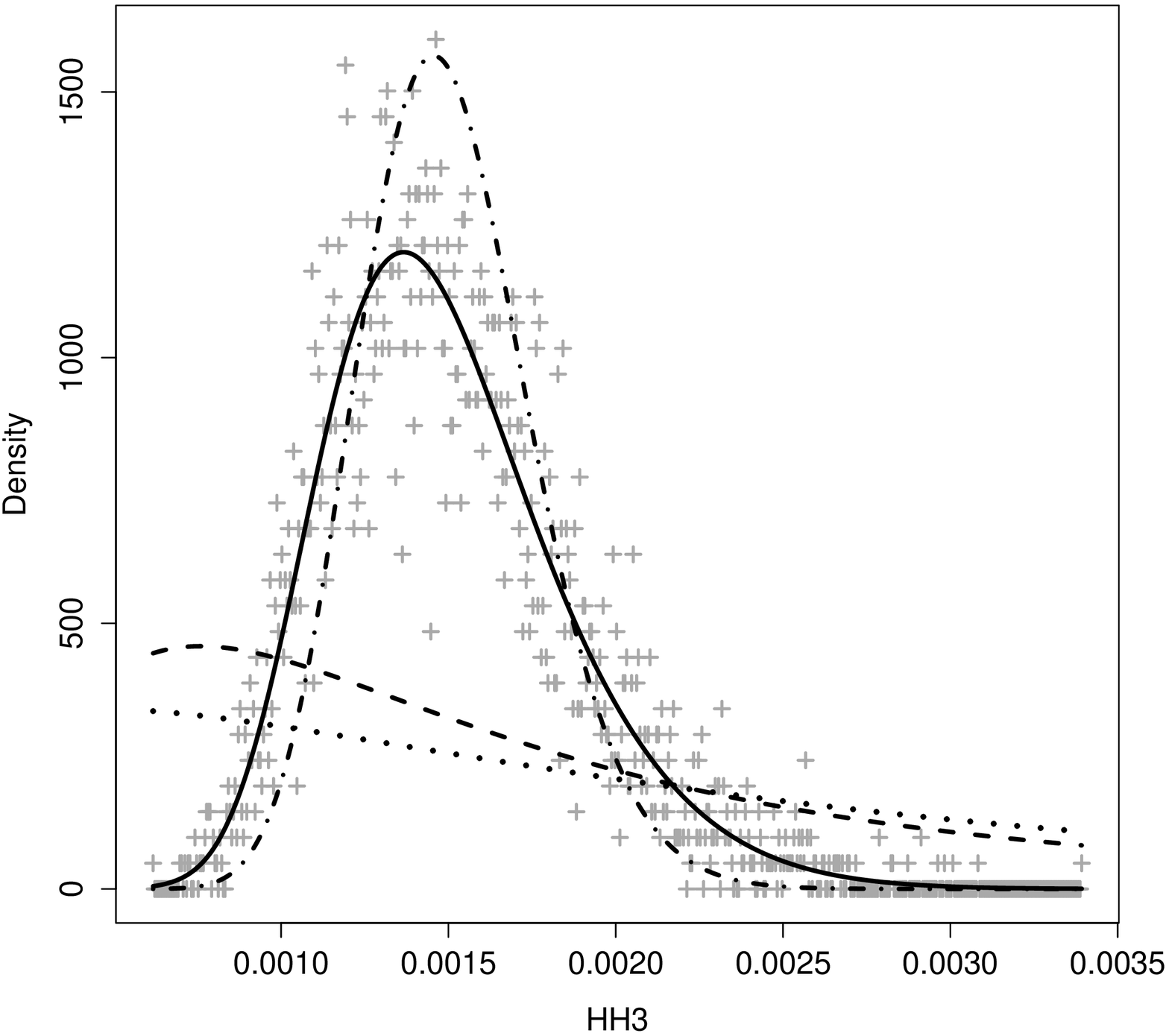}}
\subfigure[(channel,region)=(HV,A3)  \label{a32}]{\includegraphics[width=.3\linewidth]{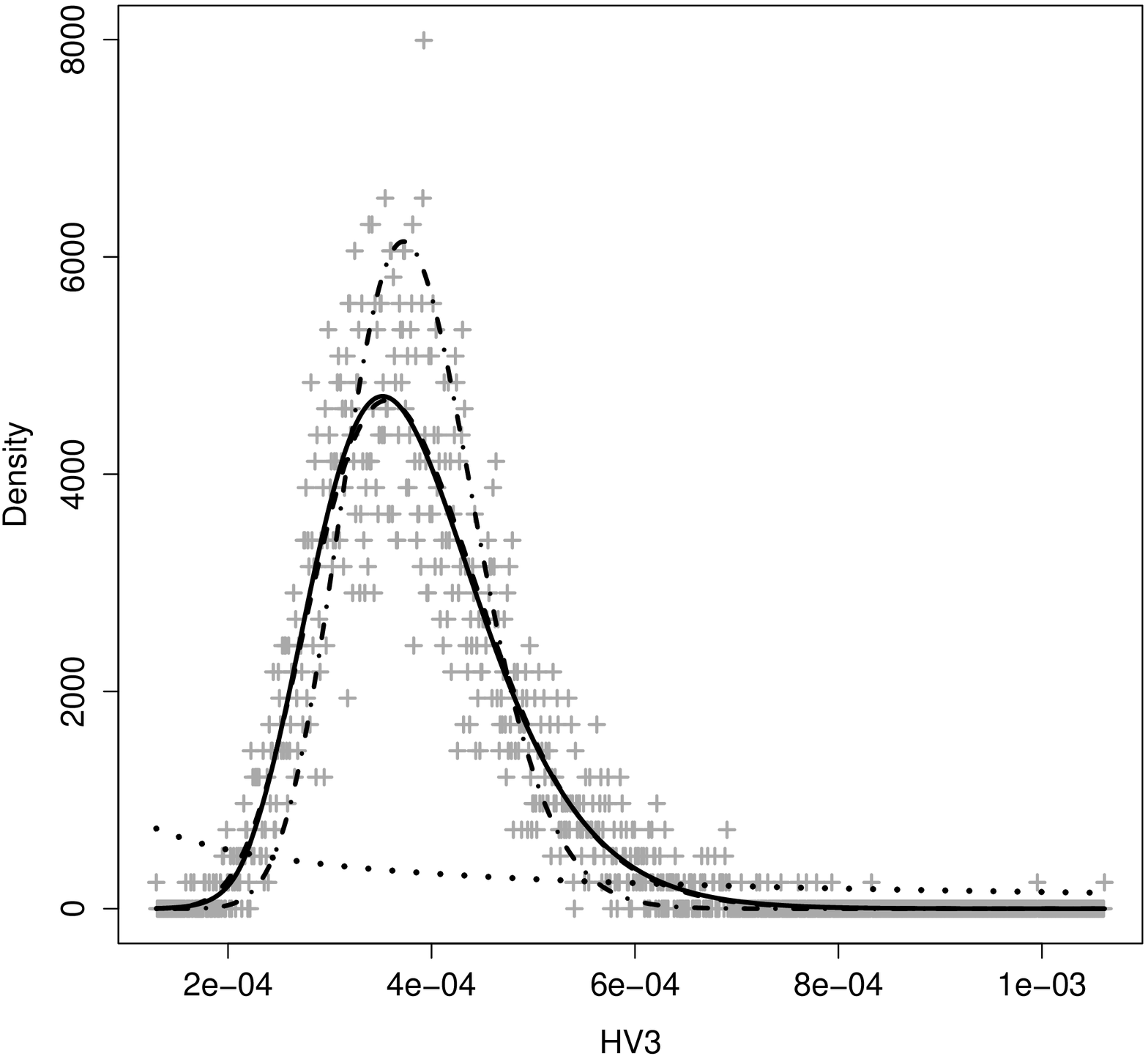}}
\subfigure[(channel,region)=(VV,A3)  \label{a33}]{\includegraphics[width=.3\linewidth]{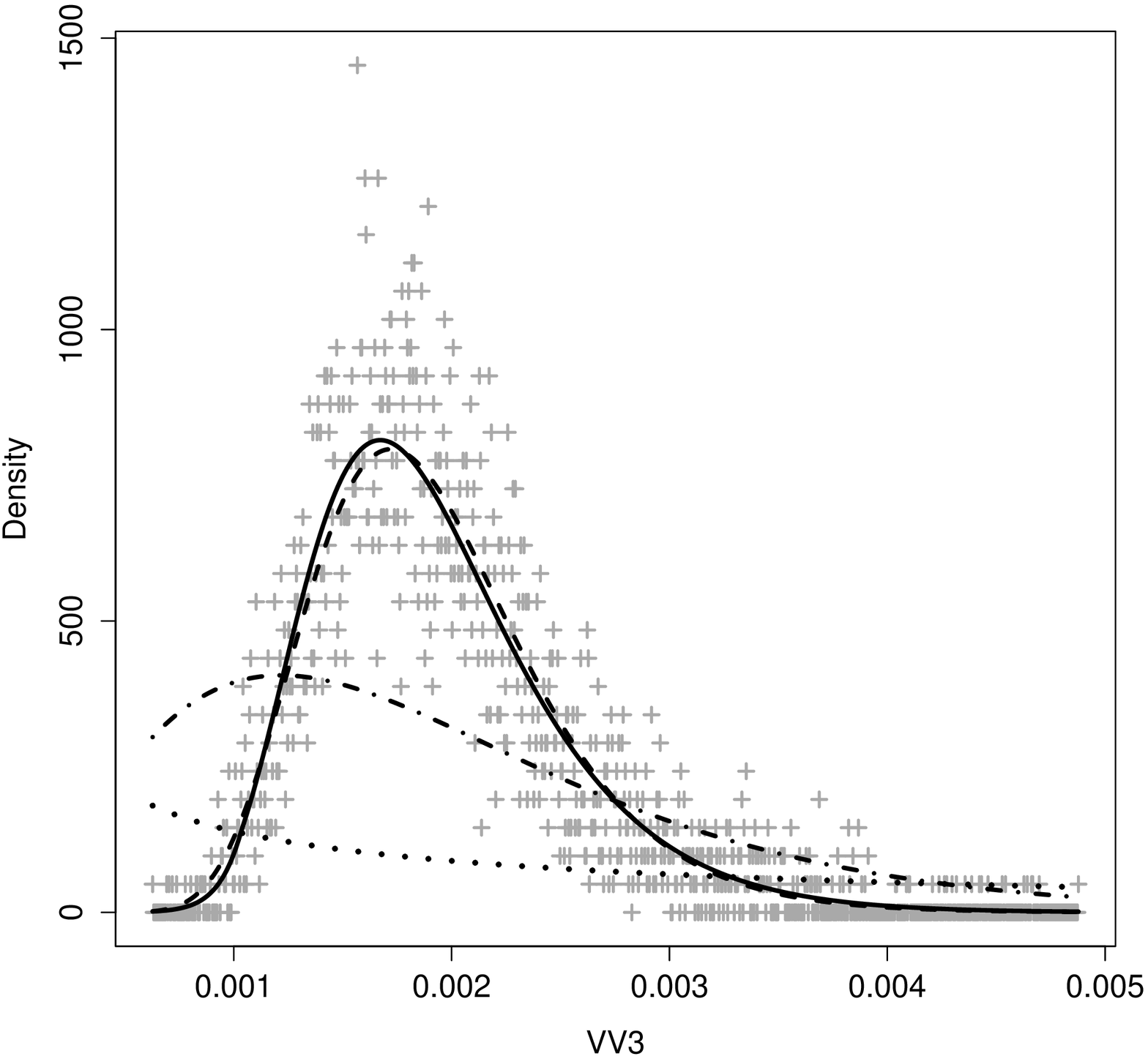}}
\caption{Plots of empirical densities (+) vs. fitted densities of BGN (solid curves), $\mathcal{G}^0$ (dashed curves), $\mathcal{K}$ (dot curves), and  $\Gamma$ (dashes and dot curves) distributions.}
\label{imagesSAR}
\end{figure}

In the context of segmentation and edge detection in images of skin,
El-Zaart and Diou~\cite{Zaart2010} presented numerical evidence indicating that the beta distribution is more adequate than the gamma distribution.
Notice that as shown in~\eqref{limit.beta} the BGN distribution collapses to the beta distribution as $s\to\infty$.

Therefore, we propose the BGN distribution as a more general model for analyzing SAR data.
We compare our results to the
$\mathcal{G}^0$~\cite{freryetal1997a},
$\mathcal{K}$~\cite{Blacknell1994},
and
$\Gamma$~\cite{Delignonetal2002} distributions,
which are regarded as classical stochastic models for SAR data.

To compare the aforementioned models,
we select three sub-images,
whose descriptive statistical measure
are displayed in Table~\ref{table2}.
Notice that the sample mean and median on actual data, $\overline{Z}$ and $\widetilde{Z}$, obey the following inequality: A2 $>$ A1 $>$ A3.
This fact characterizes A1, A2, and A3 as regions with strong, moderate, and weak returns, respectively.
In terms of standard deviation~$\operatorname{sd}(Z)$
and coefficient of variation~$\operatorname{CV}(Z)$,
region A3 presents a small degree of variability.
Table~\ref{table2} also lists the ML estimates of the BGN parameters,
 which indicate similar conclusions to those obtained in the descriptive analysis.

\begin{table}
\centering
\scriptsize
\caption{Statistical descriptive measures and ML estimates for distribution BGN parameters} \label{table2}
\begin{tabular}{c@{ }c c@{ }c@{ }c@{ }c c@{ }c@{ }c@{ }c@{ }c} \toprule
\multirow{2}{*}{Regions} & \multirow{2}{*}{channel} & \multicolumn{4}{c}{Descriptive Measures} & \multicolumn{5}{c}{ML Estimates}  \\
\cmidrule(lr{.25em}){3-6} \cmidrule(lr{.25em}){7-11}
&         &  $\overline{Z}$ ($\times 10^{-3}$) & $\widetilde{Z}$ ($\times 10^{-3}$) & $\operatorname{sd}(Z)$ &
$\operatorname{CV}(Z)$ & $\widehat{\alpha}$ & $\widehat{\beta}$ ($\times 10^{-2}$) & $\widehat{\mu}$ ($\times 10^{-3}$) &
$\widehat{\sigma}$ & $\widehat{s}$ \\ \midrule
A1 & HH &  45.42 &  42.14 &  1.96$\times 10^{-2}$ & $43.16$ & 1.30 & 2.61 &  15.16 &  1.97 &  0.50  \\
   & HV &   3.85 &   2.78 &  0.28$\times 10^{-2}$ & $73.31$ & 0.83 & 0.14 &   0.30 &  1.90 &  0.21  \\
   & VV &  88.95 &  82.06 &  4.11$\times 10^{-2}$ & $46.20$ & 1.40 & 4.03 &  23.03 &  1.71 &  0.25  \\ \midrule
A2 & HH & 204.43 & 191.20 & 10.50$\times 10^{-2}$ & $51.35$ & 1.36 & 4.87 &  57.84 &  3.16 &  0.24  \\
   & HV &  94.45 &  94.45 &  4.01$\times 10^{-2}$ & $42.50$ & 2.06 & 4.12 &  36.87 &  1.43 &  0.21  \\
   & VV &  97.01 &  95.31 &  3.84$\times 10^{-2}$ & $39.54$ & 2.00 & 3.32 &  37.20 &  2.54 &  0.21  \\ \midrule
A3 & HH &   1.50 &   1.46 &  3.59$\times 10^{-4}$ & $23.91$ & 1.66 & 0.11 &   0.30 &  1.77 &  0.31  \\
   & HV &   0.38 &   0.37 &  0.92$\times 10^{-4}$ & $23.86$ & 1.42 & 0.02 &   0.14 &  6.32 &  1.07  \\
   & VV &   1.93 &   1.84 &  5.72$\times 10^{-4}$ & $29.57$ & 1.30 & 0.10 &   0.59 &  5.40 &  0.70  \\
\bottomrule
\end{tabular}
\end{table}

Figures~\ref{a11}-\ref{a33} exhibit fitted curves and empirical densities
for all considered sub-images and distributions.
None of the classical models ($\mathcal{G}^0$, $\mathcal{K}$, and $\Gamma$)
could adequately characterize all polarization channels.
For example, considering region A3,
the $\mathcal{G}^0$ distribution
is well suited only for HH and HV polarization channels.
In contrast, the proposed model could perform well in all polarization channels.
In order to numerically compare the classical SAR modeling and the proposed BGN model,
we adopt the following measures of goodness-of-fit~\protect\cite{Gao2010,SeghouaneAmari2007}:
Akaike's information criterion (AIC),
its corrected version (AICc),
and Bayesian information criterion (BIC).
The results are presented  in Table~\ref{table3},
where best performances are in boldface.
Except for data of region A2 and polarization channel HH,
the BGN distribution could outperform all classical models.

\begin{table}
\centering
\setlength{\tabcolsep}{2pt}
\scriptsize
\caption{Goodness-of-fit measures for SAR image models based on actual data} \label{table3}
\begin{tabular}{c@{ }c c@{ }c@{ }c c@{ }c@{ }c  c@{ }c@{ }c} \toprule
\multicolumn{2}{c}{Model}  & \multicolumn{3}{c}{A1} & \multicolumn{3}{c}{A2} & \multicolumn{3}{c}{A3} \\
\cmidrule(lr{.25em}){3-5} \cmidrule(lr{.25em}){6-8} \cmidrule(lr{.25em}){9-11}
\multicolumn{2}{c}{}& AIC & AICc & BIC & AIC & AICc & BIC & AIC & AICc & BIC \\ \midrule
BGN
   & HH &  \textbf{-23493.26} & \textbf{-23491.24} & \textbf{-23486.44}  & -4648.91 & -4646.87 & -4643.22  & \textbf{-54100.65} & \textbf{-54098.63} & \textbf{-54094.00} \\
   & HV &  \textbf{-43918.24} & \textbf{-43916.23} & \textbf{-43911.43}  & \textbf{-9226.70} & \textbf{-9224.66} & \textbf{-9221.01}  & \textbf{-65430.96} & \textbf{-65428.94} & \textbf{-65424.30} \\
   & VV &  \textbf{-16703.57} & \textbf{-16701.55} & \textbf{-16696.75}  & \textbf{-9501.40} & \textbf{-9499.36} & \textbf{-9495.71}  & \textbf{-50662.42} & \textbf{-50660.40} & \textbf{-50655.77} \\ \midrule
$\mathcal{G}^0$
   & HH &  -16021.29 & -16019.28 & -16010.47  & -4645.52 & -4643.51 & -4635.83  & -47547.37 & -47545.36 & -47536.72 \\
   & HV &  -43578.39 & -43576.38 & -43567.57  & -5454.64 & -5452.63 & -5444.95  & -65430.21 & -65428.20 & -65419.56 \\
   & VV &  -13175.16 & -13173.15 & -13164.34  & -3700.03 & -3698.01 & -3690.34  & -50643.55 & -50641.54 & -50632.90 \\ \midrule
$\mathcal{K}$
   & HH &  -23389.38 & -23387.37 & -23378.56  & \textbf{-4670.55} & \textbf{-4668.53} & \textbf{-4660.86}  & -45705.15 & -45703.14 & -45694.50 \\
   & HV &  -42769.62 & -42767.61 & -42758.80  & -9135.29 & -9133.27 & -9125.60  & -48153.30 & -48151.29 & -48142.65 \\
   & VV &  -16653.93 & -16651.92 & -16643.11  & -9468.13 & -9466.11 & -9458.44  & -37370.50 & -37368.49 & -37359.85 \\ \midrule
$\Gamma$
& HH & -23377.23 & -23375.22 & -23362.41 & -4294.44 & -4292.43 & -4280.75 & -53056.29 & -53054.28 & -53041.64 \\
& HV & -21612.42 & -21610.42 & -21597.60 & -9171.60 & -9169.59 & -9157.91 & -64399.02 & -64397.02 & -64384.37 \\
& VV & -16657.36 & -16655.36 & -16642.54 & -9491.76  & -9489.76 & -9478.07 & -47268.50 & -47266.50 & -47253.85 \\
\bottomrule
\end{tabular}
\end{table}

\section{Conclusion}
\label{section-conclusion}

The proposed beta generalized normal distribution is an extension
of the generalized normal distribution previously introduced in~\cite{nadarajah2005generalized}.
We provide a comprehensive mathematical discussion for the new distribution,
which includes
shape and asymptotic behavior analysis
and
the derivation of the hazard rate function.
Power series expansions
for the moments,
for the moment generating function,
and
for the mean deviations about the mean and median
are also determined.
The method of maximum likelihood is used to estimate the model parameters.
Additionally,
we employ the derived statistical tools in the context of image processing
of radar data.
By means of Akaike's information criterion,
we show that the BGN model more adequately describes
the statistical distribution of the image pixels
from
pasture and ocean data.
The new distribution provides better fits than
the gamma model,
which is usually employed for this type of data.

\appendix

\section{Proof of Proposition~\ref{proposition-moment}}
\label{appendix-proof-proposition}

Using the generalized binomial expansion,
we have:
\begin{align*}
\mathrm{E}(X^n)&=\frac{1}{\sigma \mathrm{B}(\alpha,\beta)}
\sum_{j=0}^{\infty}
(-1)^j\binom{\beta-1}{j}\int_{-\infty}^\infty x^n
\left[\Phi_s\left(\frac{x-\mu}{\sigma}\right)
\right]^{j+\alpha-1}
\phi_s \left(\frac{x-\mu}{\sigma}\right)
\mathrm{d}x.
\end{align*}
We define the auxiliary quantity
$
h_j^{(s)}=
\left[
\Phi_s\left(
\frac{x-\mu}{\sigma}
\right)
\right]^{j+\alpha-1}
\phi_s\left(\frac{x-\mu}{\sigma}\right)
$.
Thus,
\begin{align*}
\mathrm{E}(X^n)
=\int_{-\infty}^\infty x^n g(x) \mathrm{d}x=
\frac{1}{\sigma\mathrm{B}(\alpha,\beta)}
\sum_{j=0}^{\infty}
(-1)^j\binom{\beta-1}{j}\,\int_{-\infty}^\infty x^n\,h_j^{(s)}
\mathrm{d}x.
\end{align*}
Setting $z =\frac{x-\mu}{\sigma}$ yields
\begin{align}
\nonumber
\int_{-\infty}^\infty x^n\,h_j^{(s)}
\mathrm{d}x
&=
\sigma\,\int_{-\infty}^\infty
(\sigma z+\mu)^n\,[\Phi_s(z)]^{j+\alpha-1}\,\phi_s(z)
\mathrm{d}z\\
\label{eq.section5}
&=
\sigma
\mu^n\,\sum_{i=0}^n\,\binom{n}{i}\,
\left(\frac{\sigma}{\mu}
\right)^i\,\int_{-\infty}^\infty z^i\,
[\Phi_s(z)]^{j+\alpha-1}\,\phi_s(z)\mathrm{d}z.
\end{align}
Spliting the integration range of~\eqref{eq.section5} in two
and considering
$\Phi_s(-z)= 1-\Phi_s(z)$ and
$\phi_s(z) =\phi_s(-z)$~\cite{nadarajah2005generalized}
yields
\begin{align*}
\int_{-\infty}^\infty z^i\,\phi_s(z)\,[\Phi_s(z)]^{j+\alpha-1}\mathrm{d}z
=&\int_0^\infty z^i\,\phi_s(z)\,[\Phi_s(z)]^{j+\alpha-1}\mathrm{d}z
\\
&+ \int_0^\infty (-1)^i\,z^i\,\phi_s(z)
[1-\Phi_s(z)]^{j+\alpha-1}\mathrm{d}z.
\end{align*}

From Lemma~\ref{lemApp1} in the Appendix, we can express a non-integer power of a
\mbox{cdf} in terms of a power series of this \mbox{cdf}.
We write
$
\Phi_s(z)^\alpha=\sum_{k=0}^\infty v_k(\alpha)\,\Phi_s(z)^k
$,
where
$
v_k(\alpha)=\sum_{m=k}^\infty (-1)^{k+m}\,\binom{\alpha}{m}\,\binom{m}{k}
$.
Hence, based on such expansion, we obtain
\begin{align*}
\int_{-\infty}^\infty z^i\,\phi_s(z)\,[\Phi_s(z)]^{j+\alpha-1}
\mathrm{d}z=&
\sum_{k=0}^\infty v_k(j+\alpha-1)\,\int_0^\infty z^i\,\phi_s(z)\,\Phi_s(z)^k
\mathrm{d}z
\\
&+\sum_{k=0}^\infty (-1)^{i+k}\,\binom{j+\alpha-1}{k}\,\int_0^\infty
z^i\,\phi_s(z)\,\Phi_s(z)^k\mathrm{d}z.
\end{align*}
Defining the auxiliary quantity
$
J_{i,k}^{(s)}=\int_0^\infty z^i\,\phi_s(z)\,\Phi_s(z)^k \mathrm{d}z
$,
where $i,k$ are integers, we have
\begin{align*}
\int_{-\infty}^\infty
z^i\,\phi_s(z)\,[\Phi_s(z)]^{j+\alpha-1}\,\mathrm{d}z=\sum_{k=0}^\infty
\left[v_k(j+\alpha-1)+(-1)^{i+k}\,\binom{j+\alpha-1}{k}\right]\,J_{i,k}^{(s)}.
\end{align*}

Substituting the above result into the expression for the $n$th moment,
we have proved the proposition.

\section{Evaluation of $J_{i,k}^{(s)}$}
\label{appendix-J}

First let us define the following more general
quantity
\begin{align*}
J_{i,k}^{(s)}(l_1,l_2)&=\int_{l_1}^{l_2} z^i \phi_s(z)\Phi_s(z)^k \mathrm{d}z,
\end{align*}
where $0\leq l_1\leq l_2$. Thus, $J_{i,k}^{(s)}=J_{i,k}^{(s)}(0,\infty)$.

Applying the expression for $\Phi_s(z)$ in the definition of
$J_{i,k}^{(s)}(l_1,l_2)$ and letting $y=z^s$ yields
\begin{align*}
J_{i,k}^{(s)}(l_1,l_2)
&=
\frac{1}{[2\Gamma(1/s)]^{k+1}}
\int_{l_1}^{l_2}
y^{\frac{i+1}{s}-1}\,\exp(-y)
\left[
2\Gamma(1/s)
-\Gamma(1/s,y)
\right]^k
\mathrm{d}y.
\end{align*}

The incomplete gamma function admits the power
series expansion
$
\Gamma(a,x)=\Gamma(a)-x^a\,\sum_{m=0}^\infty\frac{(-x)^m}{(a+m)\,m!}
$
as shown in~\cite{nadarajah2008order}.
Hence, using the binomial expansion, we obtain
\begin{equation}
\label{inner.integral}
\begin{split}
J_{i,k}^{(s)}(l_1,l_2)
=&
\frac{1}{[2\Gamma(1/s)]^{k+1}}
\sum_{j=0}^k
\binom{k}{j}
[\Gamma(1/s)]^{k-j}
\\
&
\times
\int_{l_1}^{l_2}
y^{\frac{i+1}{s}-1}
\exp(-y)
\left[
y^{1/s}\sum_{m=0}^\infty \frac{(-y)^m}{(1/s+m)m!}
\right]^j
\mathrm{d}y.
\end{split}
\end{equation}
By Corollary~\ref{cor:serie} in the Appendix,
we can rewrite the above
integral as
\begin{align}
\int_{l_1}^{l_2} y^{\frac{i+j+1}{s}-1}\,\exp(-y)
&
\left[
\sum_{m=0}^\infty
\frac{(-y)^m}{(1/s+m)m!}
\right]^j
\mathrm{d}y
=
\sum_{m=0}^\infty
c_{m,j}
\int_{l_1}^{l_2}
y^{m+\frac{i+j+1}{s}-1}
\exp(-y)
\mathrm{d}y
\nonumber \\
&=
\sum_{m=0}^\infty
c_{m,j}
\left[
\Gamma
\left(m+\frac{i+j+1}{s},l_1\right)
-
\Gamma
\left(m+\frac{i+j+1}{s},l_2\right)
\right]
\label{equation-integral}
,
\end{align}
where $c_{0,j}=s^j$ and
$c_{m,j}=(ms)^{-1}\sum_{r=1}^{m}(rj-m+r)\frac{(-1)^r}{(1/s+r)r!}c_{m-r,j}$
for all $m \geq 1$.

Applying~\eqref{equation-integral} in~\eqref{inner.integral}
and letting $l_1=0$ and $l_2=\infty$,
we obtain:
\begin{align}
\label{J-form-1}
J_{i,k}^{(s)}
&=
\frac{1}{[2\Gamma(1/s)]^{k+1}}
\sum_{j=0}^k\binom{k}{j}\,[\Gamma(1/s)]^{k-j}\sum_{m=0}^\infty c_{m,j}\,
\Gamma\left(m+\frac{i+j+1}{s}\right).
\end{align}

\section{Auxiliary Lemmata}
\label{appendix.simple}

\begin{lemma}
If $\alpha>0$, then
\begin{align*}
\Phi_s(z)^\alpha
=
\sum_{r=0}^\infty
\sum_{m=r}^\infty (-1)^{m+r}
\binom{\alpha}{m}\binom{m}{r}
\Phi_s(z)^r.
\end{align*}
\label{lemApp1}
\end{lemma}

\begin{proof}
In order to obtain an expansion for $\Phi_s(z)^\alpha$,
for $\alpha>0$ real non-integer,
we can write the following binomial expansion:
\begin{align*}
\Phi_s(z)^\alpha=\{1-[1-\Phi_s(z)]\}^\alpha
=
\sum_{j=0}^\infty(-1)^j\,\binom{\alpha}{j}\,[1-\Phi_s(z)]^j.
\end{align*}
Consequently, it follows that
\begin{align*}
\Phi_s(z)^\alpha
=
\sum_{m=0}^\infty\sum_{r=0}^m
(-1)^{m+r}\,\binom{\alpha}{m}\,\binom{m}{r}\,\Phi_s(z)^r.
\end{align*}
We can substitute $\sum_{m=0}^\infty\sum_{r=0}^m$ for
$\sum_{r=0}^\infty\sum_{m=r}^\infty$ to obtain
the sought result.
\end{proof}

\begin{lemma}
\begin{align*}
\left[\sum_{m=0}^\infty\frac{(-1)^m}{(1/s+m)m!}y^m \right]^j
=
\sum_{m=0}^\infty c_{m,j} y^{m},
\end{align*}
where $c_{0,j}=s^j$ and
$c_{m,j}=(ms)^{-1}\sum_{r=1}^{m}(rj-m+r)\frac{(-1)^r}{(1/s+r)r!}c_{m-r,j}$
for all $m \geq 1$.
\label{cor:serie}
\end{lemma}

\begin{proof}
Considering~\cite[Sec.~0.314]{gradshteyn2000table}
with
$a_m=\frac{(-1)^m}{(1/s+m)m!}$ the result is proved.
\end{proof}

\section*{Acknowledgements}

This work was partially supported by
the
CNPq and FACEPE,
Brazil.
Authors thank the reviewers.

\bibliographystyle{siam}%
\bibliography{bgn}%

\end{document}